\documentclass[11pt]{article}
\usepackage{graphicx, color, xcolor}
\usepackage{hyperref}

\usepackage{amsmath,amsfonts,amssymb}
\usepackage{amsthm}

\def\forall{\hbox{for all}~}
\def\L{\mathbf{L}}

\def\ve{\varepsilon}

\def\R{{\mathbb R}}

\def\implies{\Longrightarrow}
\def\vp{\varphi}

\def\P{{\cal P}}

\def\vs{\vskip 2em}

\def\v{\vskip 1em}
\def\O{{\cal O}}

\def\C{{\cal C}}
\def\F{{\cal F}}

\def\Hat{\widehat}
\def\meas{\hbox{meas}}

\def\bega{\begin{array}}
\def\enda{\end{array}}
\def\begi{\begin{itemize}}
\def\endi{\end{itemize}}

\def\ds{\displaystyle}
\def\bel{\begin{equation}\label}
\def\eeq{\end{equation}}
\def\sqr#1#2{\vbox{\hrule height .#2pt
\hbox{\vrule width .#2pt height #1pt \kern #1pt
\vrule width .#2pt}\hrule height .#2pt }}
\def\square{\sqr74}
\def\endproof{\hphantom{MM}\hfill\llap{$\square$}\goodbreak}

\newtheorem{theorem}{Theorem}[section]

\newtheorem{lemma}{Lemma}[section]
\newtheorem{example}{Example}[section]

\newtheorem{remark}{Remark}[section]
\newtheorem{definition}{Definition}[section]

\definecolor{bluegreen}{rgb}{0.0, 0.87, 0.87}
\definecolor{blush}{rgb}{0.87, 0.36, 0.51}
\definecolor{bittersweet}{rgb}{1.0, 0.44, 0.37}
\definecolor{burgundy}{rgb}{0.5, 0.0, 0.13}
\definecolor{cardinal}{rgb}{0.77, 0.12, 0.23}

\def\wen{}

\begin{document}

\title{\bf Conservation Laws with Discontinuous Gradient-Dependent Flux: the Unstable Case}
\vs

\author{Debora Amadori$^*$,  Alberto Bressan$^{**}$, and Wen Shen$^{**}$\\
\, \\
{\small  *  Dipartimento di Matematica Pura e Applicata, Universit\`a degli Studi dell'Aquila, }
\\  {\small  Via Vetoio, 67010 Coppito, Italy.}\,\\ 
{\small  ** Mathematics Department, Penn State University,
University Park, PA 16802, U.S.A. } \\ 
\,\\
{\small E-mails:  debora.amadori@univaq.it, ~axb62@psu.edu,~ wxs27@psu.edu}
\,\\   }
\maketitle  

\begin{abstract}
The paper is concerned with a scalar conservation law with discontinuous gradient-dependent flux.
Namely, the flux is described by two different functions $f(u)$ or $g(u)$, 
when the gradient $u_x$ of the solution is positive or negative, respectively.
We study here the unstable case where $f(u)>g(u)$ for all $u\in\R$.
Assuming that both $f$ and $g$ are strictly convex, solutions to the Riemann problem
are constructed. Even for a smooth initial data, examples show that the Cauchy problem can have infinitely many solutions.

For an initial data which is piecewise monotone, i.e., increasing or decreasing on a finite number of intervals, 
a solution can be constructed globally in time. It is proved that such solution is unique under the additional 
requirement that the number of interfaces, where the flux switches between $f$ and $g$, remains as small as possible.
\end{abstract}

\maketitle

\section{Introduction}
\label{sec:1}
\setcounter{equation}{0}

In this paper we consider a conservation law  with discontinuous flux, 
where the discontinuity is related to the sign of the gradient $u_x$, namely
\bel{1}
u_t + \Big[\theta(u_x) f(u) + \bigl(1-\theta(u_x) \bigr)g(u)\Big]_x~=~0.\eeq
Here $f,g$ are smooth flux functions, and $\theta$ is the step function
\bel{2}\theta(s)~=~\left\{\bega{rl} 1\quad &\hbox{if}\quad s>0,\cr
0\quad &\hbox{if}\quad s<0.\enda\right.\eeq

{\wen
Note that the function $\theta(u_x)$ is undetermined at points where
$u_x=0$. In particular, it can be arbitrary on intervals where the solution $u(t,x)$ is constant in $x$.  
We emphasize that the model \eqref{1}-\eqref{2} is entirely different from 
those found in earlier literature.  Conservation laws with discontinuous flux
have been extensively studied in the cases where the
flux function depends on time, space or the conserved quantity, i.e.,
$F=F(t,x,u)$.  Here, on the other hand, we consider a discontinuous flux $F=F(u,u_x)$ which switches
according to the sign of $u_x$.}

The equations \eqref{1}-\eqref{2}) can be used as a model for traffic flow, where $u$ is the density of cars. 
The main feature of the model lies in the key assumption that the drivers behave differently when they 
are in accelerating or decelerating mode (with their foot on the gas or on the brake pedal).
Such assumption is reasonable by personal experience as well as witnessed by traffic data \cite{TM_data}. 
One expects that, in the region where the  density decreases, i.e.~$u_x(t,x) <0$,  the cars accelerate; 
and if the density increases, the cars  decelerate.   
This leads to two different flux functions: $f$ and $g$, in the decelerating and accelerating mode, respectively.
The above equations capture this feature of  traffic flow, assuming that the switching between the acceleration 
and deceleration mode happens instantly. 
If  $\theta=\theta(u_x)$
is replaced by a smooth function, such transition would be gradual. 

In principle, there can be other models where the flux has discontinuous dependence on the gradient of the conserved quantity. 
In this paper we thus consider \eqref{1} in a general setting, not necessarily restricted to traffic flow. 

In the companion paper \cite{ABS}, under the assumption $f< g$, solutions to (\ref{1}) 
are constructed as the unique limits of parabolic approximations
\bel{3} u_t +  \Big[\eta_\ve(u_x) f'(u) + \bigl(1-\eta_\ve(u_x) \bigr)g'(u)\Big] u_x
~=~\eta'_\ve(u_x) \bigl[ g(u) - f(u)\bigr] u_{xx}\,.
\eeq
Here $\eta_\ve(s)~\doteq~\eta(s/\ve)$, while $\eta:\R\mapsto [0,1]$ is a  smooth, nondecreasing function satisfying
\bel{tprop1} \eta(s)~=~\left\{\bega{rl} 1\quad &\hbox{if}\quad s\geq 1,\cr
0\quad &\hbox{if}\quad s\leq -1.\enda\right.
\eeq
{\wen
To build some intuition, consider Fig.~\ref{f:df4} where the initial data $\bar u$ is plotted as a thick curve, 
with a local maximum at $x_1$ and a local minimum at $x_2$. 
Consider a small interval around $x_1$. On the left of $x_1$ one has $u_x >0$, hence the incoming flux is $f(u)$. 
Similarly, on the right of $x_2$ one has $u_x<0$ and  the ougoing flux is $g(u)$.  In the case where $f<g$, the total mass on this small interval decreases, and the point $x_1$ acts as a sink. A symmetric argument applies to the point $x_2$ of local minimum, which acts as a source. As a consequence, 
all local maxima decrease in time, while local minima increase.
}
When $f<g$, the discontinuity of the flux has a strong stabilizing effect. Indeed, as proved in \cite{ABS}, 
the unique limits of the viscous approximations (\ref{3}) yield a contractive semigroup w.r.t.~the $\L^1$ distance.

\begin{figure}[htbp]
\begin{center}
\resizebox{.65\textwidth}{!}{
\begin{picture}(0,0)%
\includegraphics{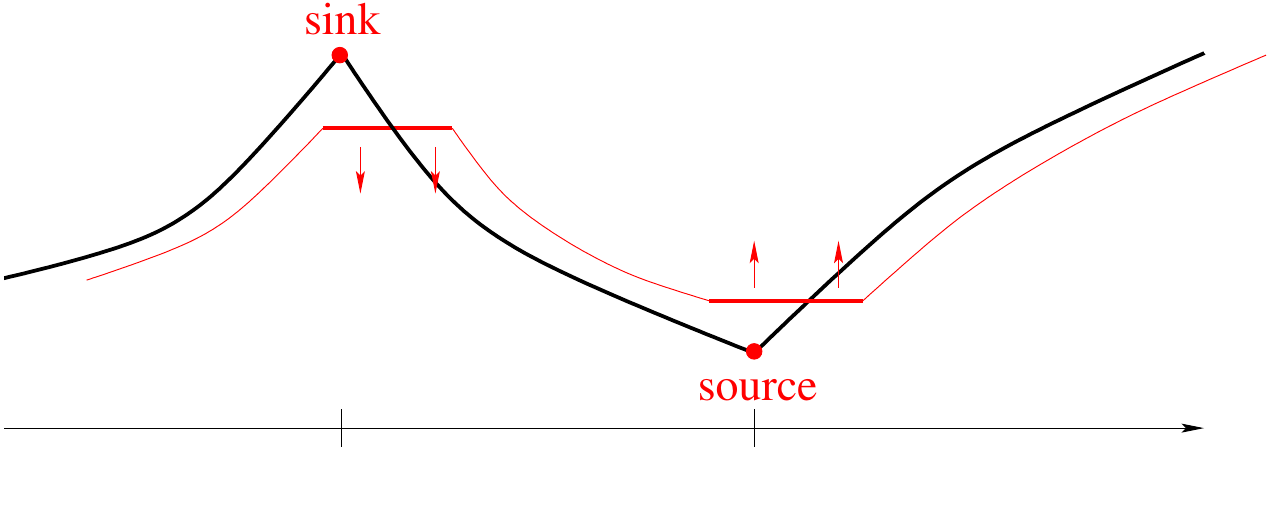}%
\end{picture}%
\setlength{\unitlength}{3947sp}%
\begin{picture}(10140,4043)(-1232,214)
\put(1401,314){\makebox(0,0)[lb]{\smash{\fontsize{20.74}{24}\usefont{T1}{ptm}{m}{n}{\color[rgb]{0,0,0}$x_1$}}}}
\put(4651,314){\makebox(0,0)[lb]{\smash{\fontsize{20.74}{24}\usefont{T1}{ptm}{m}{n}{\color[rgb]{0,0,0}$x_2$}}}}
\put(6001,3164){\makebox(0,0)[lb]{\smash{\fontsize{20.74}{24}\usefont{T1}{ptm}{m}{n}{\color[rgb]{0,0,0}$\bar u$}}}}
\end{picture}%
}
\caption{\small  A solution of (\ref{1})-(\ref{2})  in the stable case where $f< g$.  
Here the initial flux is $f(u)$ for $x<x_1$ and $x>x_2$, and $g(u)$ for $x_1<x<x_2$.
This produces a sink at $x_1$ and a source at $x_2$. }
\label{f:df4}
\end{center}
\end{figure}

Aim of the present paper is to investigate the alternative case, where
\bel{f>g}
f(u)\, >\, g(u)\qquad\qquad  \forall u\in\R\,.
\end{equation}
Notice that in this case the  equation (\ref{3}) is backward parabolic, hence ill posed.  
Therefore we refer to it as the {\it unstable case}.
In fact, {\wen using a similar intuitive argument}, here the jump in the flux
produces a source 
at every point $x_j$ of local maximum, and a sink 
at every point $x_k$ of local minimum (see Fig.~\ref{f:df5}).
In a typical situation, this further increases the total variation of the solution. 

\begin{figure}[htbp]
\begin{center}
\resizebox{.65\textwidth}{!}{
\begin{picture}(0,0)%
\includegraphics{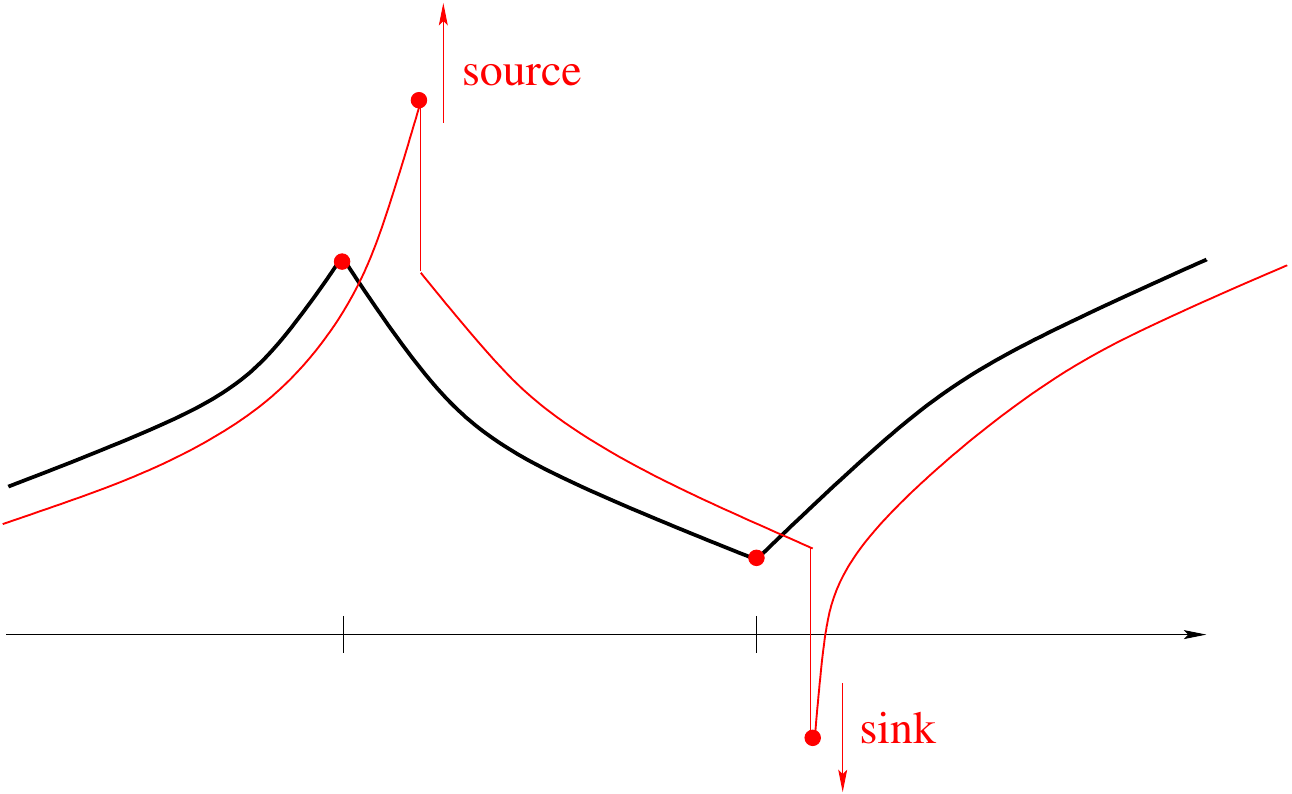}%
\end{picture}%
\setlength{\unitlength}{3947sp}%
\begin{picture}(10319,6339)(-1251,-434)
\put(6226,3164){\makebox(0,0)[lb]{\smash{\fontsize{20.74}{24}\usefont{T1}{ptm}{m}{n}{\color[rgb]{0,0,0}$\bar u$}%
}}}
\put(1376,389){\makebox(0,0)[lb]{\smash{\fontsize{20.74}{24}\usefont{T1}{ptm}{m}{n}{\color[rgb]{0,0,0}$x_1$}%
}}}
\put(4626,389){\makebox(0,0)[lb]{\smash{\fontsize{20.74}{24}\usefont{T1}{ptm}{m}{n}{\color[rgb]{0,0,0}$x_2$}%
}}}
\end{picture}%
}
\caption{\small  A solution of (\ref{1})-(\ref{2})  in the unstable case where $f> g$.  
Here the initial flux is $f(u)$ for $x<x_1$ and $x>x_2$, and $g(u)$ for $x_1<x<x_2$.
This produces a source at $x_1$ and a sink at $x_2$.}
\label{f:df5}
\end{center}
\end{figure}

Throughout this paper, we focus the analysis on the case of convex fluxes, with the following assumptions.
\begi
\item[{\bf (A1)}] {\it The functions $f$ and $g$ are $\C^2$ and strictly convex: $f''(u)\geq c_0$
and $g''(u)\geq c_0$ for some $c_0>0$ and all $u\in\R$. Moreover, $f(u)>g(u)$ for 
all $u\in\R$.}\endi

We consider the Cauchy problem with initial data
\bel{idata} u(0,x)~=~\bar u(x)\eeq
having bounded variation.

{\wen 
To introduce an appropriate definition of weak solution, we observe that for every BV function
$u:\R\mapsto\R$, the distributional derivative $\mu = D_x u$ is a signed Radon measure.
It can be decomposed into a positive and a negative part, as $\mu=\mu_+-\mu_-$.
Given a measurable function $\theta:\R\mapsto [0,1]$, the requirement that $\theta=1$ at points where 
$\mu$ is positive, while $\theta=0$ at points where $\mu$ is negative can now be stated
in terms of the measure identity
\bel{mid1}\mu ~=~\chi_{\strut\{ \theta=1\}} \mu_+ - \chi_{\strut\{\theta=0\}} \mu_-\,.\eeq
As usual, here $\chi_{\strut S}$ denotes the characteristic function of a set $S$.
In the following definition, we apply the identity (\ref{mid1}) to the measure $D_x u(t,\cdot)$, at a.e.~time $t$.
}

{\wen
\begin{definition}\label{Def:1.1} We say that a BV function $u:[0,T]\times \R \mapsto \R$ is a {\bf  weak solution}
to (\ref{1})-(\ref{2}) with initial data (\ref{idata}) if the following holds.
\begi
\item[(i)] The map $t\mapsto u(t,\cdot)$ is continuous from $[0,T]$ into $\L^1_{loc}$ and satisfies
the initial condition (\ref{idata}). 
\item[(ii)] There exists a measurable function 
$\theta: [0,T]\times \R\mapsto [0,1]$ such that
\bel{bta} D_xu(t,\cdot)~=~\chi_{\{\theta=1\}}  \Big[D_xu(t,\cdot)\Big]_+ -\chi_{\{\theta=0\}}  \Big[D_x u(t,\cdot)\Big]_-\qquad\hbox{for a.e.}~t\in [0,T],\eeq
and such that, for every compactly supported test function $\vp\in \C^1_c\bigl( ]0,T[\,\times\R\bigr)$,
\bel{wsol} 
\int_0^T \int  \left\{u\vp_t  + \Big[\theta\, f(u) + \bigl(1-\theta\, \bigr)g(u)\Big]\, \vp_x\right\}\, dxdt ~=~0.
\eeq
\endi
\end{definition}
}

{\wen 
\begin{remark}\label{r:11} 
{\rm For a.e.~$t\in [0,T]$, the function $\theta(t,\cdot) : \R\mapsto [0,1]$ that satisfies the identity in 
(\ref{bta}) is uniquely defined up to a set of $\mu$-measure zero, where $\mu = D_x u(t,\cdot)$.
On the other hand, the integral in (\ref{wsol}) is taken w.r.t.~Lebesgue measure, and it does not change if $\theta $ is modified on a set of Lebesgue measure zero. In particular, the values of $\theta(t,\cdot)$ on the jump part 
or on the Cantor part of the measure $D_x u(t,\cdot)$ are irrelevant.
In Definition~\ref{Def:1.1}, we could thus replace (\ref{bta}) with the requirement that
\bel{thprop}\theta(t,x)~=~\left\{ \bega{rl} 1\quad &\hbox{if}\quad u_x(t,x)>0,\\[2mm]
0\quad &\hbox{if}\quad u_x(t,x)<0.\enda\right.\eeq
We recall that, if $u(t,\cdot)$ has bounded variation, then the gradient $u_x(t,\cdot)$ is well defined for a.e.
$x\in\R$.     This equivalent definition of weak solution is indeed the one used in the paper~\cite{ABS}.
}
\end{remark}
}

In connection with (\ref{wsol}), to write the corresponding Rankine-Hugoniot and Lax admissibility conditions, we recall that a point 
$(\bar t, \bar x)$ is called a {\bf point of approximate jump} for the functions $u,\theta$ if there
exist values $u^-, u^+\in \R$ and $ \theta^-,\theta^+\in \{0,1\}$, together with a speed $\lambda$, such that the following holds.   
Setting
\bel{UT}U(t,x) ~\doteq~\left\{\bega{rl}u^-\quad\hbox{if}~~x<\lambda t,\\
u^+\quad\hbox{if}~~x>\lambda t,\enda\right.\qquad 
\Theta(t,x)~\doteq~\left\{\bega{rl}\theta^-\quad\hbox{if}~~x<\lambda t,\\
\theta^+\quad\hbox{if}~~x>\lambda t,\enda\right.\eeq
one has 
\begin{align}
\label{UJ} 
 \lim_{r\to 0+}~\frac{1}{ r^2} \int_{\bar t-r}^{\bar t+r} \int_{\bar x-r}^{\bar x+r} \bigl| u(t,x)- U(t-\bar t, x-\bar x)\bigr|\, dx\, dt~&=~0,\qquad\\
\label{TJ} 
\lim_{r\to 0+}~\frac{1}{ r^2} \int_{\bar t-r}^{\bar t+r} \int_{\bar x-r}^{\bar x+r} \bigl| \theta(t,x)- \Theta(t-\bar t, x-\bar x)\bigr|\, dx\, dt~&=~0.
\end{align}
In the above setting, if $u,\theta$ provide a weak solution to (\ref{1})-(\ref{2}), a standard argument \cite{Bbook}
yields that, at every point of approximate jump, the left and the right values $u^\pm, \theta^\pm$ and the speed $\lambda$ in 
(\ref{UT}) must satisfy the Rankine-Hugoniot conditions
\bel{RH} \lambda~=~\frac{\bigl[\theta^+f(u^+) + (1-\theta^+) g(u^+) \bigr] - \bigl[\theta^-f(u^-) + (1-\theta^-) g(u^-) \bigr]
}{ u^+-u^-}\,.\eeq
\begin{definition} Let $u,\theta$ be a weak solution to (\ref{1})-(\ref{2}), and let $(\bar t,\bar x)$ be a point of approximate jump.   
We say that this solution satisfies the {\bf Lax admissibility conditions} if 
the left and right states $u^\pm, \theta^\pm$ satisfy the inequalities
\bel{Lax} 
\theta^+f'(u^+) + (1-\theta^+) g'(u^+)  ~\leq~ \lambda~\leq~ \theta^-f'(u^-) + (1-\theta^-) g'(u^-).
\eeq
\end{definition}

Roughly speaking, the main results of this paper show that:
\begi
\item[(i)] Even  for a smooth initial data, infinitely many piecewise smooth admissible  solutions can be constructed.
\item[(ii)] Given the initial profile of $\theta(\cdot)$ at time $t=0$, a unique solution can be selected by further requiring that 
the number of interfaces (i.e.~the number of points
where $\theta$ is discontinuous, 
see Definition~\ref{def:F})
remains as small as possible at $t=0+$ and all future times.
\endi

Scalar conservation laws $ u_t + f_x=0$ where the flux function $f(t, x, u,u_x)$ 
has discontinuous dependence on its arguments have several applications. 
Starting with the work by Gimse and Risebro~\cite{GR91, GR92}, conservation laws with flux function 
discontinuous in space and time have attracted wide attention, becoming the subject of an extensive literature. 
A complete list of references or a comprehensive review are outside the scope of this paper.
Here  we refer to 
\cite{AKR11, AM15, AP05, BV06,  BGS18, CR05, GNPT2006, GS19, KR99, KR95, Panov10, SeguinVovelle03, WS15, WS16, WS18, Towers2000} for a partial list.
See  also  the comprehensive survey paper \cite{And15}. 
For the case where the flux function is discontinuous w.r.t.~the conserved quantity $u$, 
there is also a rich literature, see in particular~\cite{ACnew, BGMS11, BGS13, BGS17, Gimse93, Panov23, Towers2020}.
Finally, as an example where the flux function also depends on $u_x$, 
one can consider hysteretic traffic flow where different flux functions are employed in 
the acceleration and deceleration modes. 
Some models of this kind have been proposed and analyzed in \cite{CF19, Fan24},
where stop-and-go waves are observed in the solutions.
We remark that our model \eqref{1}-\eqref{2} could also be applied in connection with traffic flow.
However, our main goal is to develop a general  framework, valid  for any couple of flux functions $(f,g)$.

{\wen 
The remainder of the paper is organized as follows. 
In Section~\ref{sec:2}, we first prove
some lemmas on admissible jumps. Then we provide  some counterexamples
showing the non-uniqueness of solutions to the Cauchy problem for (\ref{1})-(\ref{2}).
Indeed, new ``spikes" can originate from a smooth solution, 
at arbitrary points in space and time. 
Section~\ref{sec:3} provides solutions to the Riemann problem, for all possible cases of Riemann data.
In Section~\ref{sec:4} we introduce a class of piecewise monotone solutions $u=u(t,x)$, where
$u(t,\cdot)$ is increasing or decreasing on a finite set of open  intervals $I_k(t)=\,\bigl]y_{k-1}(t), y_k(t)\bigr[\,$.  
For an initial data $\bar u$ of this type, Theorem~\ref{t:41} states the global existence of solutions.
 A unique solution can be singled out, imposing the additional requirement that the number of interfaces $y_k(\cdot)$ 
remains as small as possible.  The proof of the theorem is  worked out
in Section~\ref{sec:5}.   Toward this goal, one has to patch together solutions to the conservation laws
\bel{fgclaw} u_t+f(u)_x\,=\,0,\qquad\hbox{or}\qquad u_t+g(u)_x\,=\, 0,\eeq
defined on the various intervals $I_k(t)$.  Each interface
$y_k(t)$ is here determined by solving a particular ODE with possibly discontinuous right hand side.
The construction  relies on a uniqueness result for 
discontinuous ODEs, originally developed in~\cite{B88} and later applied to conservation laws in \cite{BS98}.
Finally, in Section~\ref{sec:6} we hint at some future research directions,
including the emergence of stop-and-go waves as a result of the instability 
of the solutions, in response to small, random perturbations.
We emphasize that the lack of uniqueness of solutions should not be regarded as a drawback of the present model.
In fact, it opens up the possibility of developing new stochastic models,
by introducing a probability distribution on the family of all
piecewise monotone solutions. As in the forthcoming paper \cite{BSu}, this can be done by
allowing new interfaces to arise at random points in space and time.
}

\section{Admissible jumps and examples of non-uniqueness}
\label{sec:2}
\setcounter{equation}{0}

{\wen 
In this section we consider a weak solution of (\ref{1}) and characterize the 
family of jumps that satisfy the Lax condition (\ref{Lax}).  For a conservation law with a single convex flux,
this admissibility condition is sufficient to single out a unique solution to 
the Cauchy problem.
However, this is not the case in the present setting with two fluxes, as it will be shown by counterexamples.
}

We start by proving that the Lax condition in \eqref{Lax} rules out upward jumps. 
\begin{lemma}\label{lm:Jump1}
Let the flux functions $f,g$ satisfy the assumptions {\bf (A1)}. Then any upward jump
violates the Lax admissibility conditions~\eqref{Lax}. 
\end{lemma} 

\noindent\textbf{Proof.}
Let $u^-,u^+$ denote the left and right state of the jump.
If $\theta^-=\theta^+$, the conclusion follows from the standard theory of conservation laws with strictly convex flux, where  upward jumps are not admissible. 
When $\theta^-\not=\theta^+$,  two cases must be considered, 
as shown in Fig.~\ref{f:df41}, left. 
\begin{itemize}
\item 
If $\theta^-=1$ and $\theta^+=0$, then $u^-$(and $u^+$)  
lies on the graph of $f$ (and $g$) resp.  Assume $u^-<u^+$. 
Since $f,g$ satisfy {\bf (A1)}, then the speed of the jump is smaller than 
the characteristic speed on its right $g'(u^+)$, violating the Lax condition~\eqref{Lax}. 
\item 
On the other hand, if $\theta^-=0$ and $\theta^+=1$, then $u^-$(and $u^+$)  
lies on the graph of $g$ (and $f$) resp. 
The jump speed is larger than the characteristic on its left $g'(u^-)$,
violating the Lax condition~\eqref{Lax}.  
\endproof
\end{itemize}

\begin{figure}[htbp]
\begin{center}
\resizebox{.8\textwidth}{!}{
\begin{picture}(0,0)%
\includegraphics{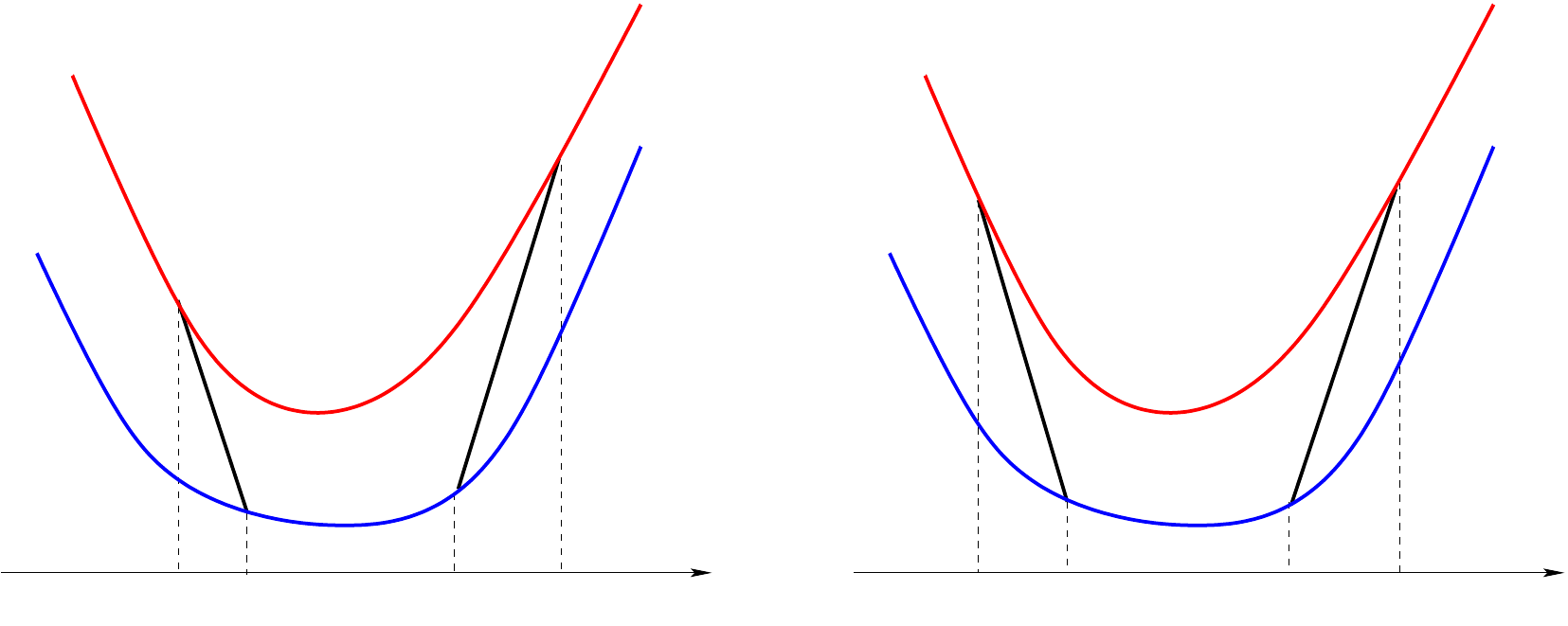}%
\end{picture}%
\setlength{\unitlength}{3947sp}%
\begin{picture}(13224,5278)(1189,-20006)
\put(6261,-17201){\makebox(0,0)[lb]{\smash{\fontsize{20.74}{26.4}\usefont{T1}{ptm}{m}{n}{\color[rgb]{0,0,1}$g$}%
}}}
\put(5421,-15716){\makebox(0,0)[lb]{\smash{\fontsize{20.74}{26.4}\usefont{T1}{ptm}{m}{n}{\color[rgb]{1,0,0}$f$}%
}}}
\put(5776,-19901){\makebox(0,0)[lb]{\smash{\fontsize{20.74}{26.4}\usefont{T1}{ptm}{m}{n}{\color[rgb]{0,0,0}$u^+$}%
}}}
\put(4761,-19871){\makebox(0,0)[lb]{\smash{\fontsize{20.74}{26.4}\usefont{T1}{ptm}{m}{n}{\color[rgb]{0,0,0}$u^-$}%
}}}
\put(3121,-19871){\makebox(0,0)[lb]{\smash{\fontsize{20.74}{26.4}\usefont{T1}{ptm}{m}{n}{\color[rgb]{0,0,0}$u^+$}%
}}}
\put(2421,-19856){\makebox(0,0)[lb]{\smash{\fontsize{20.74}{26.4}\usefont{T1}{ptm}{m}{n}{\color[rgb]{0,0,0}$u^-$}%
}}}
\put(12786,-15791){\makebox(0,0)[lb]{\smash{\fontsize{20.74}{26.4}\usefont{T1}{ptm}{m}{n}{\color[rgb]{1,0,0}$f$}%
}}}
\put(13416,-17261){\makebox(0,0)[lb]{\smash{\fontsize{20.74}{26.4}\usefont{T1}{ptm}{m}{n}{\color[rgb]{0,0,1}$g$}%
}}}
\put(12826,-19871){\makebox(0,0)[lb]{\smash{\fontsize{20.74}{26.4}\usefont{T1}{ptm}{m}{n}{\color[rgb]{0,0,0}$u^-$}%
}}}
\put(11826,-19856){\makebox(0,0)[lb]{\smash{\fontsize{20.74}{26.4}\usefont{T1}{ptm}{m}{n}{\color[rgb]{0,0,0}$u^+$}%
}}}
\put(10081,-19871){\makebox(0,0)[lb]{\smash{\fontsize{20.74}{26.4}\usefont{T1}{ptm}{m}{n}{\color[rgb]{0,0,0}$u^-$}%
}}}
\put(9171,-19886){\makebox(0,0)[lb]{\smash{\fontsize{20.74}{26.4}\usefont{T1}{ptm}{m}{n}{\color[rgb]{0,0,0}$u^+$}%
}}}
\end{picture}%
}
\caption{\small Left: when $\theta^-\not=\theta^+$, upward jumps can never satisfy the Lax
admissibility conditions (\ref{Lax}). Right: when $\theta^-\not= \theta^+$, 
there can also be downward jumps that do not satisfy (\ref{Lax}).}
\label{f:df41}
\end{center}
\end{figure}

\begin{remark}\label{r:}
{\rm We remark that, when $\theta^-\not=\theta^+$, 
there can also be some downward jumps  that do not satisfy the admissibility conditions (\ref{Lax}).  Assuming $u^->u^+$,  as illustrated in Fig.~\ref{f:df41}, right, the conditions  (\ref{Lax})
can be violated in two cases:
\begin{itemize}
\item
When $\theta^-=0$ and $\theta^+=1$, and  $u^+$ is a point such that 
$f'(u^+)$ is bigger than the jump speed. 
\item
When $\theta^-=1$ and $\theta^+=0$, and  $u^-$ is a point such that 
$f'(u^-)$ is smaller than the jump speed. 
\end{itemize}
}\end{remark}

The next Lemma provides the admissibility conditions on the downward jumps.

\begin{lemma}\label{Lm:Jump2}
Let $u^->u^+$ and $\theta^-\not=\theta^+$. 
The downward jumps are Lax admissible in the following two cases:
\begin{itemize}
\item 
When $\theta^-=1$, $\theta^+=0$ and $u^-\ge u^*$, where
$(u^*,f(u^*))$ is the unique  point where the straight line through $(u^+,g(u^+))$ 
touches the graph of $f$ tangentially and $u^*>u^+$.
\item
When $\theta^-=0$, $\theta^+=1$ and  $u^+\le v^*$, where 
$(v^*, f(v^*))$ is the unique point where the straight line through $(u^-,g(u^-))$ 
touches the graph of $f$ tangentially and $v^*<u^-$.
\end{itemize}
\end{lemma}

\begin{proof}
The proof  is immediate, by verifying the Lax condition~\eqref{Lax}.
\end{proof}
\v

{\wen
The remainder of this section contains various examples where the Cauchy problem has multiple
solutions, all satisfying the Lax admissibility conditions.
}

\begin{figure}[htbp]
\centering
  \includegraphics[scale=0.4]{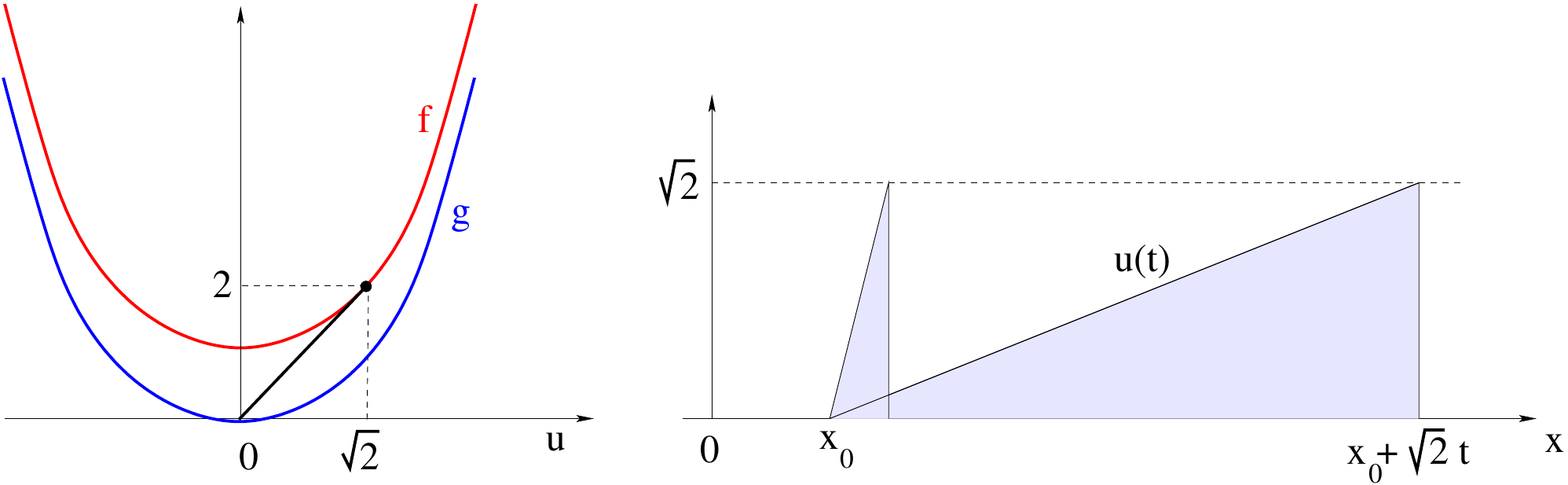}
    \caption{\small The solution described at (\ref{bss1}).}
\label{f:df6}
\end{figure}

{\wen
\begin{example} {\rm 
Consider the conservation law \eqref{1}-\eqref{2} with the two fluxes
\bel{F12}
f(u)\,=\,\frac{u^2}{ 2} +1,\qquad\qquad g(u) \,=\, \frac{u^2}{ 2}\,,\eeq
and with  initial condition
\bel{u=0} u(0,x)\,=\,\bar u(x)\,=\,0.
\eeq
Note that, since $\bar u_x=0$, the corresponding value $\bar\theta(x)\in[0,1]$ can be arbitrary.
Here we show examples of different choices of the initial function $\bar\theta$ that lead to distinct solutions:

(i) first, by choosing  $ \bar\theta(x) \equiv 0$, we obtain the trivial solution
$$u(t,x)\equiv 0, \qquad\qquad \theta(t,x)\equiv 0;$$

(ii) next, choosing an arbitrary point $x_0\in \R$ and setting
$\bar\theta(x)=1$ for $x\le x_0$ while $\bar\theta(x)=0$ for $x_0>0$, 
an  explicit solution (see Fig.~\ref{f:df6}) is found to be
\bel{bss1}
u(t,x)=\left\{ \bega{cl} \ds (x-x_0)/ t \quad &\hbox{if} \quad x\in \bigl[x_0, \,x_0+\sqrt 2 \,t
\bigr],\\[1mm]
0 \quad &\hbox{otherwise,}\enda\right.\eeq
\bel{bss11}\theta(t,x)=\left\{ \bega{cl} \ds 1 \quad &\hbox{if} ~~ x<x_0+\sqrt 2 \,t,\\[1mm]
0 \quad &\hbox{if}~~ x>x_0+\sqrt 2 \,t. \enda\right.\eeq
Indeed, on the region where $x<x_0 +\sqrt 2 \,t$, the function  $u(t,\cdot)$ is increasing 
and provides a solution to $u_t+ f(u)_x=0$.
Moreover, for $x\geq x_0 + \sqrt 2 \,t$ the function  $u(t,\cdot)$ is decreasing.
The jump at $x=y(t)\doteq x_0 + \sqrt 2 \,t$ satisfies the Rankine-Hugoniot equation
$$\dot y(t)~=~\frac{f\bigl(u(t, y(t)-) \bigr) - g\bigl(u(t, y(t)+) \bigr)}{ 
u(t, y(t)-) - u(t, y(t)+)}~=~\sqrt 2;$$

(iii) yet another family of solutions is obtained by setting $\bar\theta(x)=0$ for $x\le x_0$ and $\bar\theta(x)=1$ for $x_0>0$, 
and we get
\bel{bss2}
u(t,x)~=~\left\{ \bega{cl} \ds (x-x_0)/t \qquad &\hbox{if} \quad x\in \bigl[x_0-\sqrt 2 \,t\,,~x_0
\bigr],\\[1mm]
0 \qquad &\hbox{otherwise},\enda\right.\eeq
\bel{bss22}\theta(t,x)=\left\{ \bega{cl} \ds 1 \quad &\hbox{if} ~~ x>x_0-\sqrt 2 \,t,\\[1mm]
0 \quad &\hbox{if}~~ x<x_0-\sqrt 2 \,t. \enda\right.\eeq

}\end{example}
}

\bigskip


{\wen 
The next example shows that, even when both initial data $(\bar u, \bar \theta)$ are assigned, 
multiple solutions can still arise. Indeed, new spikes can appear in a neighborhood of any point 
where a solution $u(t,\cdot)$  is continuous and decreasing.

\begin{example}\label{ex:22}
{\rm Consider again the two fluxes $f,g$ in (\ref{F12}), and the initial data
\bel{idex2} \bar u(x)~=~\left\{ \bega{rl} 1\quad &\hbox{if}\quad x<-1,\\
-x\quad &\hbox{if}\quad x\in [-1,1],\\
-1\quad &\hbox{if}\quad x>1.\enda\right.
\qquad \qquad
 \bar\theta(x)=0\,.\eeq
For $t\in [0,1[\,$, a  solution to the Cauchy problem with initial data $\bar u$ can be constructed by the method of characteristics:
\bel{idex3} 
u_1(t,x)~=~\left\{ \bega{cl} 1\quad &\hbox{if}\quad x<-1+t,\\[1mm]
\frac{x}{ t-1}\quad &\hbox{if}\quad x\in \left[ \frac{-1}{ 1-t},\, \frac{1}{ 1-t}\right],\\[1mm]
-1\quad &\hbox{if}\quad x>1-t.\enda\right.\qquad\qquad \theta(t,x)=0.\eeq

Next, on a small initial time interval $t\in [0,t_0]$ we construct an additional solution $u_2=u_2(t,x)$ to the same Cauchy problem 
by inserting a pair of  spikes in a neighborhood of the origin. Let
$$\bigl(v_1, f(v_1)\bigr)\,=\, \left(-\sqrt 2, 2\right),\qquad\qquad \bigl(v_2, f(v_2)\bigr)\,=\, \left(\sqrt 2, 2\right)$$
 be the points where the tangent lines to the graph of $f$ pass through
$\bigl(u^*, g( u^*)\bigr)=(0,0)$.
Our second solution will have  the form
\bel{idex4} u_2(t,x)\,=\,\left\{ \bega{cl} u_1(t,x)   \quad &\hbox{if}\quad |x|> z(t),\\[1mm]
x/t 
\quad &\hbox{if}\quad |x|\leq z(t),\enda\right.\qquad
\theta(t,x)\,=\,\left\{ \bega{cl}
{\wen 0}\quad &\hbox{if} \quad |x|\ge z(t),\\[1mm]
{\wen 1}\quad &\hbox{if} \quad |x|< z(t).
\enda\right.\eeq
Here $\pm z(t)$ are the positions of the two shocks. 
{\wen Their speeds} are determined by the Rankine-Hugoniot equation
\bel{RHz}\bega{rl}
\dot z(t)&\ds=~\frac{f(z/t) - g( z/(t-1)) }{ z/t ~ - ~ z/(t-1) } 
~=~ \frac{(z^2+2t^2) (1-t)^2 - z^2 t^2}{ 2  z  t (1-t)}\,.
\enda\eeq
Since $z(0)=0$, setting $z(t) = \sigma(t)\,t$, a leading order expansion yields
$$t\, \dot\sigma(t)+ \sigma~=~\frac{\sigma}{ 2} + \frac{1}{ \sigma}+\O(1)\cdot t.$$
Letting $t\to 0+$ this yields $\sigma(0)=\sqrt 2$. As shown in Figure~\ref{f:df80} the solution 
thus satisfies
$$z(t)~=~\sqrt{2} \, t + \O(1)\cdot t^2, \qquad \qquad u_2\bigl(t, z(t)-\bigr)~= ~\sqrt 2 + \O(1) \cdot t.$$

\begin{figure}[htbp]
\centering
  \includegraphics[scale=0.4]{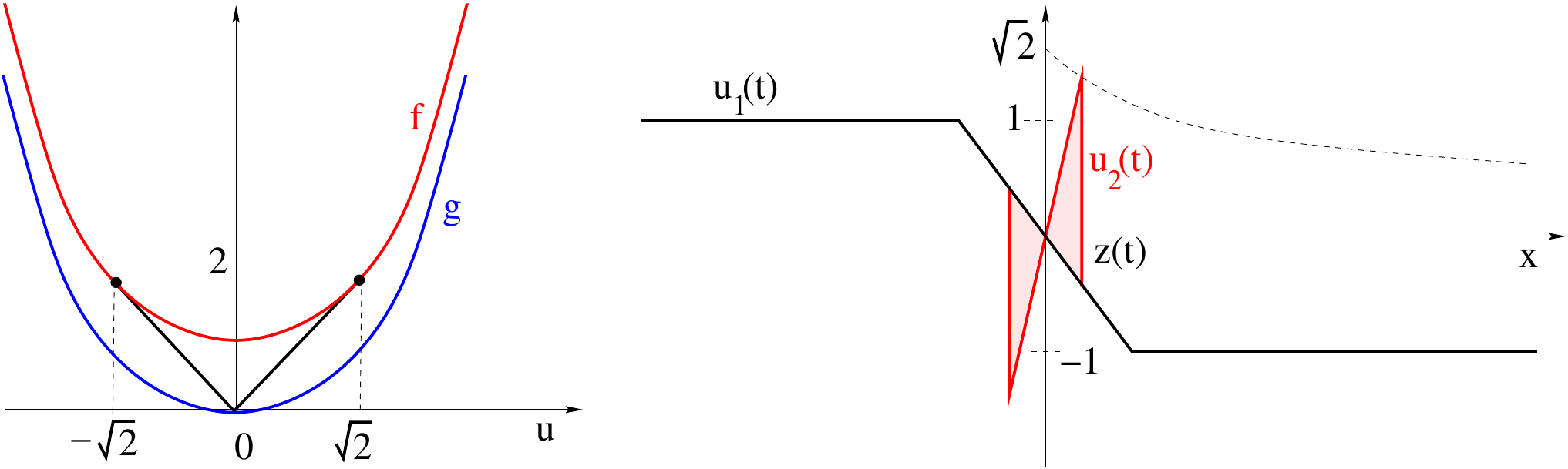}
 \caption{\small Left:  The tangent lines from the origin to the graph of $f$.
Right: the solutions $u_1$ at (\ref{idex2}) and $u_2$ at (\ref{idex3}).}
\label{f:df80}
\end{figure}

On the other hand, we observe that in the case where $u_x>0$ the above construction breaks down, because there are no
points $\bigl(v_i, g(v_i)\bigr)$, $i=1,2$, where the tangent line to the graph of $g$  goes through a point $(u^*, f(u^*))$.
}
\end{example}}

\section{Solutions to the Riemann problems}
\label{sec:3}
\setcounter{equation}{0}

{\wen
Before constructing the solutions of Riemann problems, 
we first introduce the concept of {\it interface}. 
The uniqueness result that will be stated in Theorem~\ref{t:41} 
relies on a crucial assumption on the 
number of interfaces.}


\begin{definition}\label{def:F}
Let $u=u(x)$ be a piecewise monotone function, and let $\theta=\theta(x)$ be a piecewise constant function. 
We say that a set of points
$$y_1\,<\,y_2\,<~\cdots~< y_N$$
is a set of {\bf interfaces} associated with $u$ and $\theta$  
if it contains all the jumps in $\theta$
and,
for every open interval $J_i=\,] y_i,\,  y_{i+1}[\,$, $i=0,\ldots,N$, there holds:
\begi
\item If $\theta(x)=1$ on $J_i$, then 
$u$ is monotone increasing. Namely
$$ y_i<x<x'< y_{i+1}\qquad\implies\qquad  u(x)\leq  u(x').$$
\item If $\theta(x)=0$ on $J_i$, then 
$ u$ is monotone decreasing. Namely
$$ y_i<x<x'< y_{i+1}\qquad\implies\qquad  u(x)\geq  u(x').$$
\endi
For notational convenience, we here set $y_0\doteq -\infty$, $y_{N+1}\doteq +\infty$.
%
\end{definition}

\medskip

{\wen 
\begin{remark}\label{rmk2}{\rm 
Note that the set of interfaces includes all the points where $\theta$ is discontinous.
In particular, it includes all the downward jumps in $u$ where $\theta=1$ on both sides. 
However, this set does not include the downward jumps in $u$ where $\theta=0$ on both sides. }
\end{remark}
}

In Section~\ref{sec:4}  we will show that a unique solution can be selected by requiring that 
the number of interfaces remains as small as possible.

\bigskip

We now describe a general procedure to construct a solution for piecewise constant initial data:
\bel{RD} u(0,x)\,=\,\left\{\bega{rl} u^-\quad \hbox{if}\quad x<0,\\[1mm]
u^+\quad \hbox{if}\quad x>0,\enda\right.\qquad
 \theta(0,x)\,=\,\left\{\bega{rl} \theta^-\quad \hbox{if}\quad x<0,\\[1mm]
\theta^+\quad \hbox{if}\quad x>0.\enda\right.\eeq
Notice that, since the initial data is constant for $x<0$ and $x>0$,  for any given $u^-, u^+$,
we can arbitrarily choose $\theta=0 $ or $\theta=1$ on these half lines.
We will thus consider four main cases, depending on the choice of $\theta^-,\theta^+\in\{0,1\}$.
\v
CASE 1: $\theta^-=\theta^+=1$. 

In this case, for any $u^-, u^+$, a solution to the Riemann 
problem is obtained by letting  $u=u(t,x)$ be the solution to the scalar conservation law
\bel{clawf} u_t+f(u)_x~=~0\eeq
with the same Riemann data. 
{\wen
 Since the flux $f$ is convex, we have 2 sub-cases.

CASE 1A:   $u^-< u^+$. 
In this case,  the solution of (\ref{clawf}) contains a centered rarefaction fan,  and 
\bel{the1}\theta(t,x)=1\qquad \hbox{for all}~~ t,x.\eeq    
In this case, no interfaces are present. 

We remark that CASE 1A can occur at $t=0$, if the initial datum has an upward jump. 
However, it cannot occur as the result of any wave interaction at a later time $t>0$.

\medskip

CASE 1B: $u^->u^+$.  In this case the solution $u$ 
has a single downward jump that satisfies the Lax condition \eqref{Lax}. 
Here the jump is an interface. We set 
\bel{the1B}\theta(t,x)=\begin{cases}1 &\hbox{if}~~ x\not = \lambda t\,,\\
0 &\hbox{if}~~  x=\lambda t\,,
\end{cases}\eeq
where $\lambda$ is the Rankine-Hugoniot speed of the shock.
With this choice of $\theta$, the condition (\ref{bta}) in Def.~\ref{Def:1.1} is satisfied.  
}

\begin{figure}[htb]
\centering
  \includegraphics[scale=0.35]{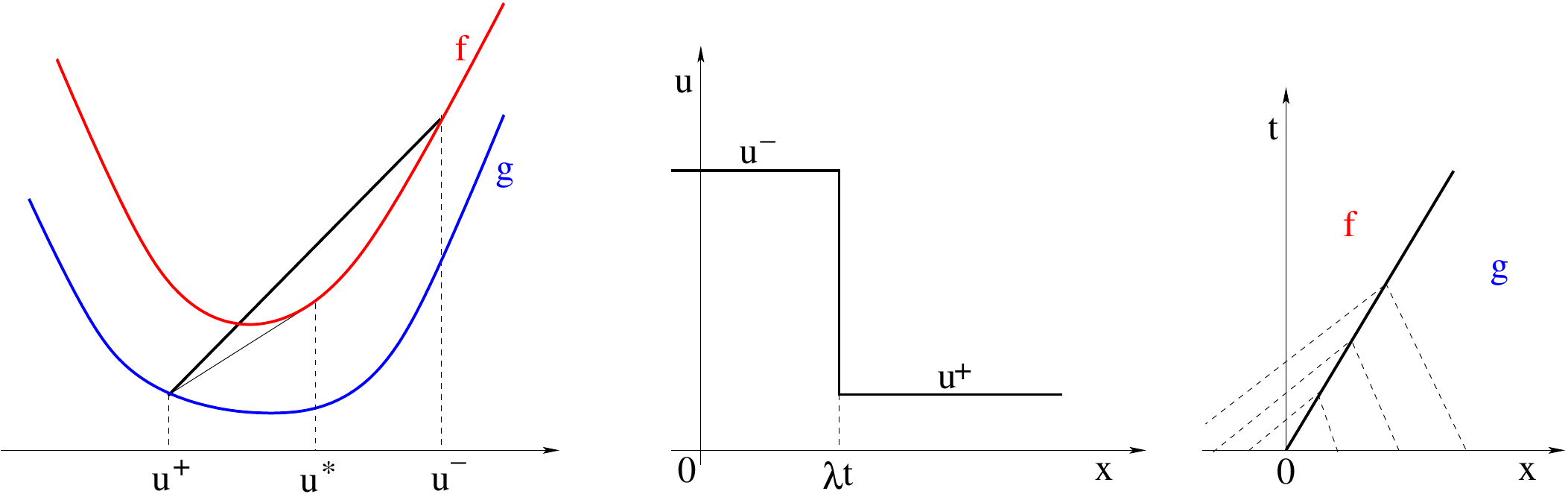}
    \caption{\small Solution of the Riemann problem in Case 2A. Left: the states $u^+<u^*<u^-$. 
    Center: the solution $u=u(t,x) $ at time $t>0$.  Right: the characteristics in the $t$-$x$ plane.
    Here the flux is $f$ or $g$ respectively to the left and to the right of the shock located at $x=\lambda t$.}
\label{f:case2A}
\end{figure}

\begin{figure}[htb]
\centering
  \includegraphics[scale=0.35]{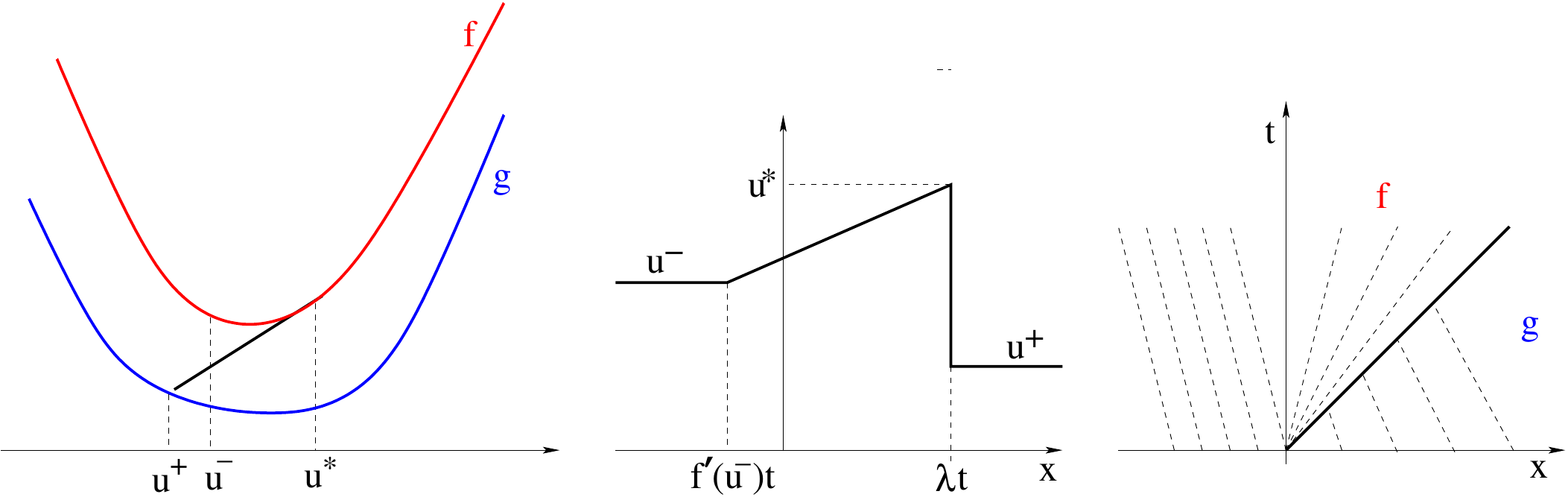}
    \caption{\small Solution of the Riemann problem in Case 2B. Left: the states $u^+<u^-<u^*$. 
    Center: the solution $u=u(t,x) $ at time $t>0$.  Right: the characteristics in the $t$-$x$ plane.
    Here the flux is $f$ or $g$ respectively to the left and to the right of the shock located at $x=\lambda t$.}
\label{f:case2B}
\end{figure}

\v
CASE 2: $\theta^-=1$, $\theta^+=0$.   

As a preliminary, consider the  solution $u=u^\flat(t,x)$ of the Cauchy problem
\bel{cp1}
u_t + f(u)_x~=~0,\qquad \qquad u(0,x)\,=\,\left\{ \bega{cl} u^-\quad&\hbox{if}~~x<0,\\[1mm]
+\infty\quad&\hbox{if}~~x>0.\enda\right.\eeq
Notice that this solution contains a single, infinitely large rarefaction. It can be implicitly defined as
\bel{uf1}
u^\flat(t,x)~=~\left\{ \bega{ll} u^-\qquad &\hbox{if}\quad x/t\leq f'(u^-),\\[1mm]%
w\qquad &\hbox{if}\quad x/t= f'(w)>f'(u^-).\enda\right.\eeq

Next, as shown in Fig.~\ref{f:case2A} (left), let 
$(u^*, f(u^*))$ be the unique point where the straight line through $(u^+, g(u^+))$ touches the graph of $f$ 
tangentially, and $u^*>u^+$.   
Two sub-cases must be considered.
\v
CASE 2A: $u^*< u^-$  (see Fig.~\ref{f:case2A}).  In this case, the solution consists of a single shock:
$$u(t,x)~=~\left\{\bega{rl} u^- \quad \hbox{if}\quad x<\lambda t,\\[1mm]
u^+\quad \hbox{if}\quad x>\lambda t,\enda\right.
\qquad\qquad
  \theta(t,x)~=~\left\{\bega{rl} 1\quad \hbox{if}\quad x<\lambda t,\\[1mm]
0\quad \hbox{if}\quad x\ge \lambda t,\enda\right.
$$
where the shock speed is
\bel{la1}\lambda~=~\frac{f(u^-) - g(u^+)}{ u^--u^+}\,.\eeq
 By Lemma~\ref{Lm:Jump2},  this jump is Lax admissible. 
{\wen  Note that characteristics impinge on the shock transversally,
 both from the left and from the right.} 
\v
CASE 2B: $u^-\leq  u^*$  (see Fig.~\ref{f:case2B}).  
In this case, the solution consists of a centered rarefaction $u^\flat$ on the left,
right next to a shock with left and right states $(u^*,u^+)$. 
We have
$$u(t,x)~=~\left\{\bega{cl} u^\flat(t,x)\qquad \hbox{if}\quad x<\lambda t,\\[1mm]
 u^+ 
\qquad \qquad \hbox{if}\quad x >\lambda t,\enda\right.$$
where the shock travels with the speed 
\bel{la2}{ \lambda}
~=~\frac{f(u^*) - g(u^+)}{ u^*-u^+}~=~f'(u^*)\,.\eeq
By Lemma~\ref{Lm:Jump2} the jump is Lax admissible.
{\wen  Note that now  the characteristics impinge on the shock transversally from the right,  
but are  tangent to the shock from the left.  }

In both cases 2A and 2B the location of the jump is an interface, and we have
\bel{thela}\theta(t,x)~=~\left\{\bega{rl} 1\qquad\hbox{if}~~x<\lambda t,\\[1mm]
0\qquad\hbox{if}~~x \ge \lambda t.\enda\right.\eeq

\begin{figure}[htb]
\centering
  \includegraphics[scale=0.35]{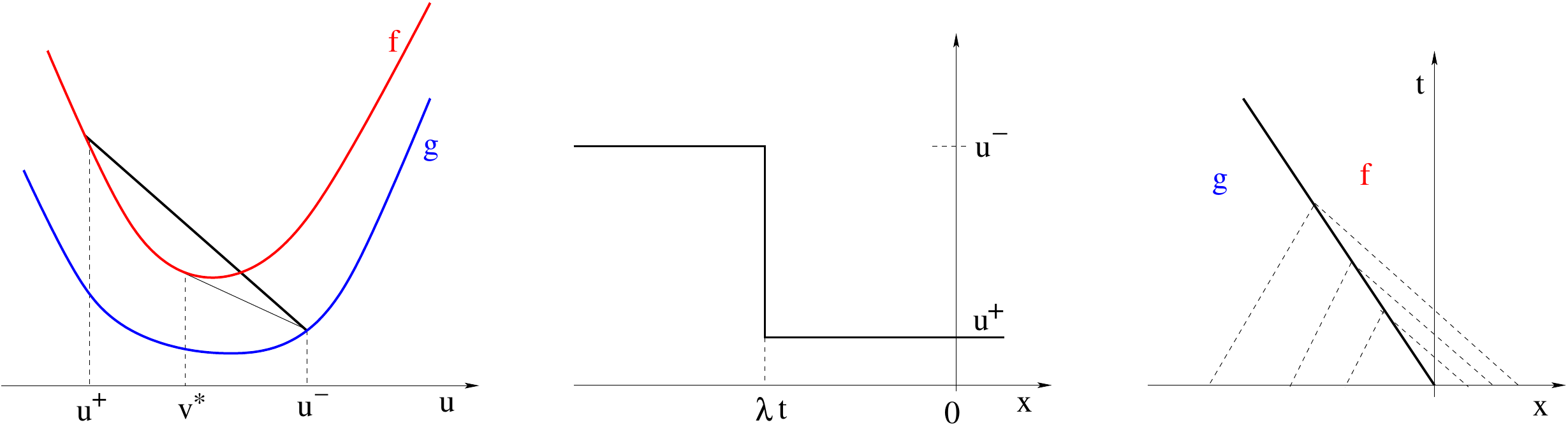}
    \caption{\small Solution of the Riemann problem in Case 3A. Left: the states $u^+<v^*<u^-$. 
    Center: the solution $u=u(t,x) $ at time $t>0$.  Right: the characteristics in the $t$-$x$ plane.
    Here the flux is $g$ or $f$ respectively to the left and to the right of the shock located at $x=\lambda t$.}
\label{f:case3A}
\end{figure}

\begin{figure}[htb]
\centering
  \includegraphics[scale=0.35]{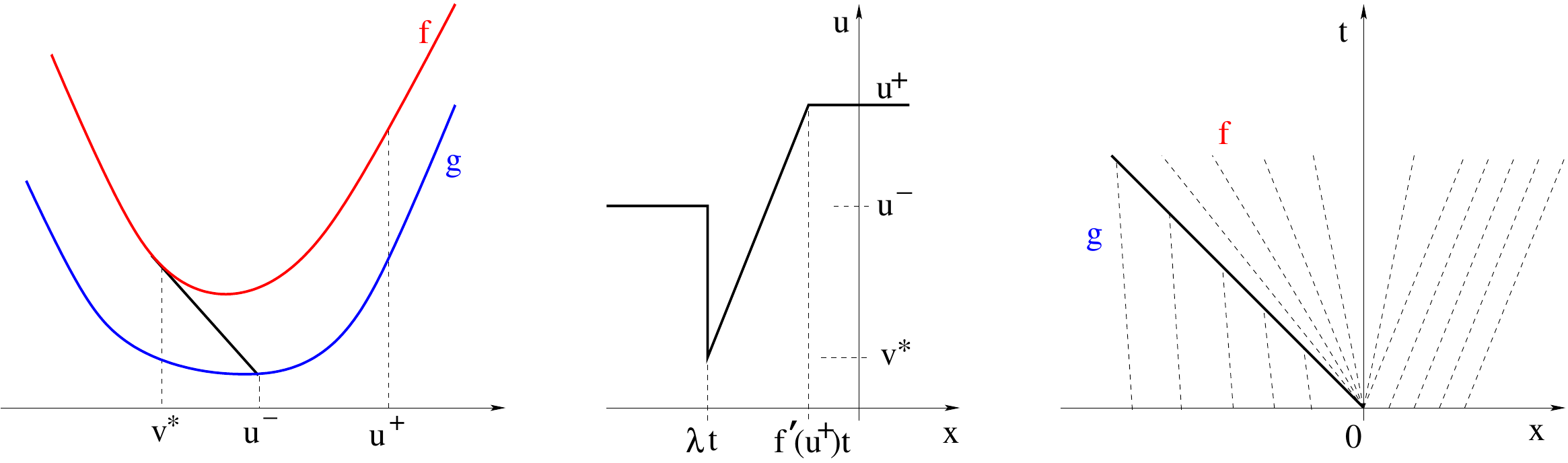}
    \caption{\small Solution of the Riemann problem in Case 3B. Left: the states $u^+<v^*<u^-$. 
    Center: the solution $u=u(t,x) $ at time $t>0$.  Right: the characteristics in the $t$-$x$ plane. 
    Here the flux is $g$ or $f$ respectively to the left and to the right of the shock located at $x=\lambda t$.}
    \label{f:case3B}
\end{figure}

CASE 3: $\theta^-=0$, $\theta^+=1$.   

{\wen This case is symmetric to CASE 2.}  
As a preliminary, consider the  solution $u=u^\sharp(t,x)$ of the Cauchy problem
\bel{cp11}
u_t + f(u)_x~=~0,\qquad \qquad u(0,x)\,=\,\left\{ \bega{cl} -\infty\quad&\hbox{if}~~x<0,\\[1mm]
u^+\quad&\hbox{if}~~x>0.\enda\right.\eeq
Notice that this solution contains a single, infinitely large rarefaction. It can be implicitly defined as
\bel{ug1}
u^\sharp(t,x)~=~\left\{ \bega{ll} 
w\qquad &\hbox{if}\quad x/t= f'(w)<f'(u^+),\\[1mm]
u^+\qquad &\hbox{if}\quad x/t\geq f'(u^+).
\enda\right.\eeq

Next, as shown in Fig.~\ref{f:case3A} (left), let 
 $(v^*, f(v^*))$ be the unique point where the straight line
through $(u^-, g(u^-))$ touches the graph of $f$ 
tangentially, and $v^*<u^-$. 
Two sub-cases must be considered.
\v
CASE 3A: $u^+< v^*$ (see Fig.~\ref{f:case3A}). 
In this case, the solution consists of a single shock:
%
$$
u(t,x)~=~\left\{\bega{rl} u^-\quad \hbox{if }~ x<\lambda t,\\[1mm]
u^+\quad \hbox{if }~ x>\lambda t,\enda\right.
\qquad 
\mbox{where}\quad 
\lambda~=~\frac{f(u^+) - g(u^-)}{ u^+-u^-}\,.
$$
{\wen  Note that characteristics impinge on the shock transversally, both from the left and from the right.}

CASE 3B: $u^+\geq v^*$ (see Fig.~\ref{f:case3B}). In this case, the solution consists of 
a centered rarefaction and a shock:
$$
 u(t,x)~=~\left\{\bega{cl} u^- \quad &\hbox{if }~ x<\lambda t,\\[1mm]
 u^\sharp(t,x)\quad &\hbox{if }~ x>\lambda t,\enda\right.
 \qquad
\mbox{where} \quad \lambda=\frac{f(v^*) - g(u^-)}{ v^*-u^-}=f'(v^*)\,.
$$
%
%
{\wen  
 Note that  the characteristics impinge on the shock transversally from the left,  
but are  tangent to the shock from the right.  }

In both cases, the jump is Lax admissible by Lemma~\ref{Lm:Jump2},  and its location is 
an interface. 
We have
$$\theta(t,x)~=~\left\{\bega{rl} 0\quad &\hbox{if}~~x \le \lambda t,\\[1mm]
 1\quad &\hbox{if}~~x>\lambda t.\enda \right.$$
\v
CASE 4: $\theta^-=\theta^+=0$. We study two sub-cases. 

CASE 4A: $u^-> u^+$. 
In this case, a solution to the Riemann 
problem is obtained by letting  $u=u(t,x)$ be the solution to the scalar conservation law
\bel{clawg} u_t+g(u)_x~=~0\eeq
with the same Riemann data, and taking 
\bel{the2} \theta(t,x)=0\qquad \hbox{for all}~~ t,x.\eeq    
Indeed, if $u^-\geq u^+$ the solution of (\ref{clawg}) is monotone decreasing  (piecewise constant with a single downward jump).
Hence $u_x(t,x)\leq 0$ at a.e.~point $(t,x)$ and the choice (\ref{the2}) satisfies (\ref{bta}).
However, note that this jump is not an interface. 

\v
CASE 4B: $u^-< u^+$.
Notice that in this case the admissible solution to (\ref{clawg}) contains a centered rarefaction,
hence it cannot be regarded as a solution to (\ref{1}). 
For this reason, we consider the points $(v^*, f(v^*))$, $( u^*, f(u^*))$ where the 
lines through the points $(u^-, g(u^-))$ and $(u^+, g(u^+))$ are tangent to the graph of $f$
(see Fig.~\ref{f:case4B}, left), with $v^*<u^-$ and $u^*>u^+$. 
As shown in Fig.~\ref{f:case4B}, center, the solution to the Riemann problem consists of a centered $f$-rarefaction, 
enclosed between two shocks. It can be implicitly defined as
\bel{RP2B} 
u(t,x)~=~\left\{ \bega{cl} u^-\quad &\hbox{if}\quad x/t < f'(v^*),\\[1mm]
w\quad &\hbox{if}\quad  
x/t = f'(w),\quad   f'(v^*) <x/t <f'(u^*),\\[1mm]
 u^+\quad &\hbox{if}\quad x/t > f'(u^*).\enda\right.\eeq
Here we can take
\bel{RP2BT}\theta(t,x)~=~\left\{ \bega{cl} 1\quad &\hbox{if} ~~ f'(v^*) <x/t <f'(u^*),\\[1mm]
 0\quad &\hbox{otherwise}.\enda\right.\eeq
Notice that the choice of $(v^*, u^*)$ implies that both jumps satisfy the Rankine-Hugoniot equation (\ref{RH}). 
Moreover they are admissible by Lemma~\ref{Lm:Jump2}. 
Both jumps 
are interfaces. 
The behavior of characteristics, to the left and to the right of these two shocks, 
is illustrated in Fig.~\ref{f:case4B}, right.
%

\begin{figure}[htbp]
\centering
  \includegraphics[scale=0.35]{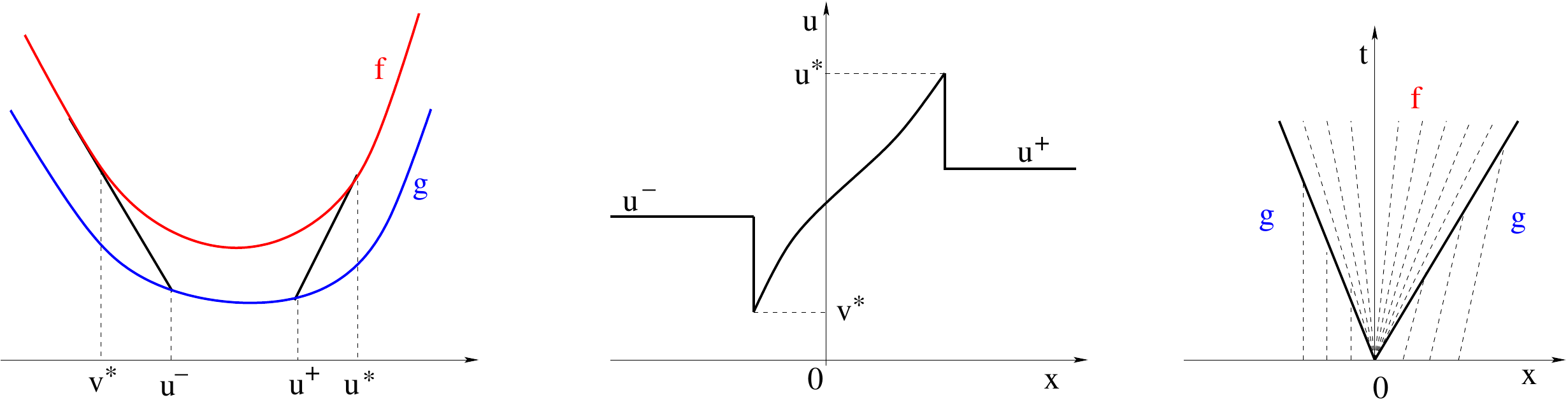}
    \caption{\small Solution of the Riemann problem in Case 4B. Left: the points $v^*, u^*$ are determined by constructing 
lines through the points $(u^-, g(u^-))$ and $(u^+, g(u^+))$ which are tangent to the graph of $f$.
Middle: a graph of the solution at time $t=1$. Right: the characteristics in the $t$-$x$ plane.
  }
    \label{f:case4B}
\end{figure}

{\wen
We remark that, similar to CASE 1A, here CASE 4B can only occur at $t=0$ with certain initial data,
but it will not occur as the result of any wave interaction at a later time $t>0$.
For instance, assume that the two interfaces $y(t)$ and $z(t)$ interact at a time $t=\tau$, with $\theta(t,x)=1$
for $y(t)<x<z(t)$,  while $\theta(t,x)=0$ for $x\leq y(t)$ and for $x\geq z(t)$, as long as $t<\tau$.  
Since $y(\tau)=z(\tau)$,  the speeds of the interfaces must satisfy $\dot y(\tau-)\ge \dot z(\tau-)$.
Since both interfaces satisfy the Lax conditions, this implies $f'(u(\tau, y(\tau)-))\ge f'(u(\tau, y(\tau)+))$. By the monotonicity of $f'$ we conclude that $u(\tau, y(\tau)-)\ge u(\tau, y(\tau)+)$, which rules out \textbf{Case 4B}.

\begin{remark}\label{rem:2*}{\rm
Taking the limit of CASE 4B  as $u^+ \rightarrow u^-$, the initial data approach 
the one in CASE 4A.   However,  the corresponding solutions of the Riemann problems approach a different limit.
This yields two admissible solutions for the case where $\theta^-=\theta^+=0$ and $u^+=u^-$.
In CASE 4A we obtain a constant solution, while as a limit of CASE 4B we obtain a  solution with
two spikes.  This confirms the instability of the model.}
\end{remark}

}

\begin{figure}[htbp]
\centering
  \includegraphics[scale=0.35]{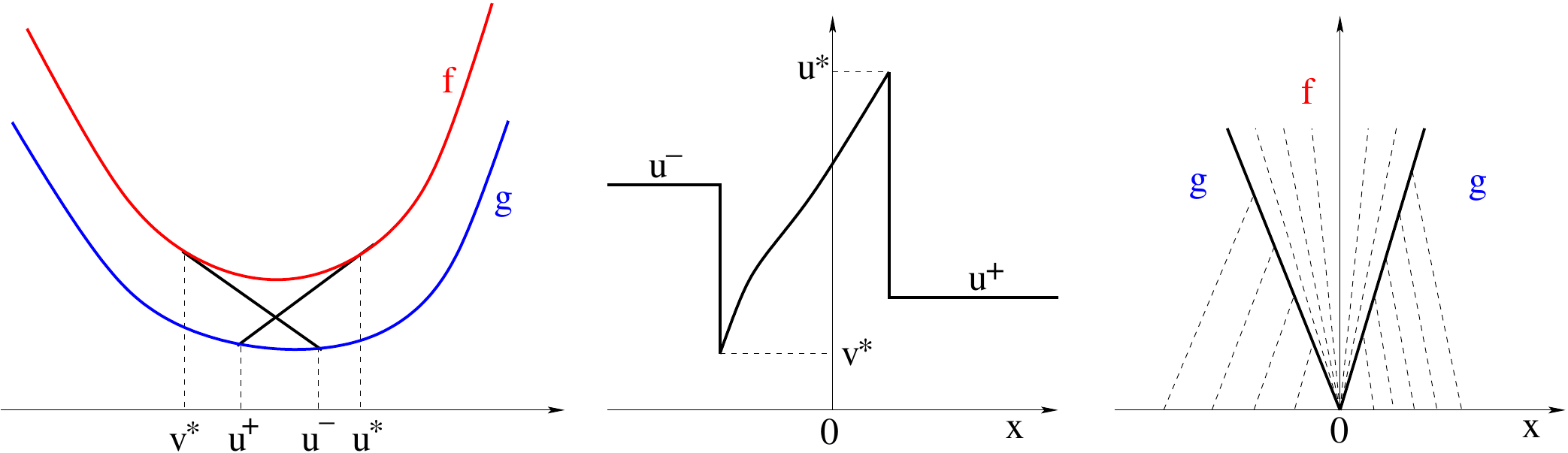}
    \caption{\small  A second admissible solution to the Riemann problem considered 
    in CASE 4A, containing two interfaces. 
  }
    \label{f:case4AAA}
\end{figure}

{\wen
\begin{remark} \label{r:21} {\rm As shown in Fig.~\ref{f:case4AAA}, left,
it is possible to have $u^->u^+$ but still $v^*<u^*$. 
In this case, one can again construct a solution to the Riemann problem of the form 
(\ref{RP2B})-(\ref{RP2BT}).  This is a second admissible solution, 
 different from the one 
constructed in CASE~4A where $\theta(t,x)=0$ for all $t,x$ and no interfaces were present.
This second solution contains two interfaces.

We thus have an example of a Riemann problem with two distinct admissible solutions.
To choose between the two, one may give preference to the one
with the least number of interfaces, namely the solution considered in CASE~4A.  }
\end{remark} 
}

\section{Piecewise monotone solutions to the Cauchy problem}
\label{sec:4}
\setcounter{equation}{0}
Throughout the following we assume that the two fluxes $f,g$  satisfy the assumptions {\bf (A1)}.
In connection with the conservation law with discontinuous flux (\ref{1})-(\ref{2}), we consider initial data comprising of
\begi
\item[{\bf (ID1)}] {\it An initial profile
\bel{ubar} u(0,x)~=~\bar u(x),\eeq
where $\bar u\in \L^\infty(\R)$ is a piecewise monotone function.}

\item[{\bf (ID2)}]  {\it A piecewise constant function $\bar\theta:\R\mapsto \{0,1\}$, 
and a finite set of  interfaces 
$$\bar y_1\,< \, \bar y_2 \,< ~\cdots~< \,\bar y_N$$
associated with $\bar u$ and $\bar\theta$, as in Definition~\ref{def:F}. 
}
\endi


We remark that the set of interfaces $\bar y_j$ is an integral part of the initial data.
If $\bar u$ is constant over an interval $[a,b]$, there are infinitely many ways to choose the 
points $\bar y_j$ according to {\bf (ID2)}. Different choices lead to distinct solutions
of the conservation law (\ref{1}).


\v
\begin{definition}\label{def:31} Let the fluxes $f,g$ satisfy {\bf (A1)}. 
We say that $u:[0,T]\mapsto\L^1_{loc}(\R)$ is a {\bf piecewise monotone admissible solution} 
to the Cauchy problem (\ref{1})-(\ref{2}) with initial data $\bar u$, $\bar \theta$, $\{\bar y_i\}$ 
as in {\bf (ID1)-(ID2)}, if the following holds.
\begi
\item[(i)] The map $t\mapsto u(t,\cdot)$ is continuous from $[0,T]$ into $\L^1_{loc}(\R)$,
and satisfies the initial condition (\ref{ubar}).  
{\wen 
Furthermore, $\forall t \in [0,T]$, $u(t,\cdot) $ is piecewise monotone.}

\item[(ii)] There exists a function $\theta=\theta(t,x)\in \{0,1\}$, with the following properties. 
The map $t\mapsto \theta(t,\cdot)$ is continuous with values in $\L^1_{loc}(\R)$, 
and satisfies the initial condition $\theta(0,x)=\bar\theta(x)$.

Moreover, for every $t\in [0,T]$ there exists a finite set of interfaces (see Definition~\ref{def:F})
$$y_1(t)\,<\,y_2(t)<~\cdots~<\,y_{\strut N(t)}(t)$$
associated with $u(t,x)$ and $\theta(t,x)$.
\item[(iii)] The conservation equation is satisfied in distributional sense:
\bel{claw}
\int_0^T \left\{
\int_{\{ \theta=1\}}  \bigl[ u \phi_t+  f(u) \phi_x\bigr]\, dx
+ \int_{\{ \theta=0\}}  \bigl[ u \phi_t+  g(u) \phi_x\bigr]\, dx\right\}dt~=~0
\eeq
for every test function $\phi\in \C^1_c\bigl( \,]0,T[\,\times\R\bigr)$.
\item[(iv)] For a.e.~time $t$, at every point $(t,x)$ where $u$ has an approximate jump the Lax
admissibility condition (\ref{Lax}) holds. 
\endi
\end{definition}

{\wen
As we already remarked in connection with the Riemann problem, 
uniqueness does not hold
even within the class of Lax-admissible, piecewise smooth solutions.     
 A possible way to single out a unique solution is to impose a minimality condition on the number of interfaces.

\begin{definition}\label{def:min} We say that a Lax-admissible, piecewise monotone solution $u$ satisfies 
the \textbf{minimal interface condition}  if, in a forward neighborhood of every point $(t,x)$, the solution
$u$ contains a minimal number of interfaces. 
\end{definition}

Notice  that, if $\theta(t,\cdot)$ is constant on an open interval $\,]a,b[$, the above condition implies that
 no new interface can be generated at any point $(t,x)$ with $a<x<b$.    In particular, this rules out 
 the second solution (\ref{idex4}) considered in Example~\ref{ex:22}.

We can now state 
our main result, on  the global existence of solutions to the Cauchy problem
with piecewise monotone initial data as in {\bf (ID1)-(ID2)}.   

\begin{theorem}\label{t:41} Let the flux functions $f,g$ satisfy {\bf (A1)}. Then,
for any initial data as in  {\bf (ID1)-(ID2)}, the Cauchy problem has a piecewise monotone admissible solution, defined for all $t\geq 0$.  Among all such solutions, a unique one is singled out by the minimal interface condition.
%
\end{theorem}
}
\v

\section{Global existence of piecewise monotone solutions}
\label{sec:5}
\setcounter{equation}{0}

To prove Theorem~\ref{t:41},  in this section we construct a piecewise monotone solution to the Cauchy problem (\ref{1})-(\ref{2}), with initial data as in {\bf (ID1)-(ID2)}.
In Section~\ref{S4-1}
we start by constructing a solution on a small time interval.
Because of the finite propagation speed, it suffices to construct a local 
solution in a neighborhood of a given point $(t_0,x_0)$.   
Without loss of generality we shall assume that this point is the origin.
Afterwards, in Section~\ref{S4-2} we show that the solution can be prolonged for all positive times.


\subsection{Construction of local solutions} \label{S4-1}
In this subsection, 
we 
assume that $\bar u$ is monotone separately on the left and on the right of the origin.  
The initial data are
\bel{ida}\left\{\bega{l} u(0,x)\,=\, \bar u(x),\\[1mm]
\theta(0,x)\,=\, \bar\theta(x),\enda\right.\quad \hbox{with} \quad \lim_{x\to \pm 0} \bar u(x)\,=\,u^\pm,\qquad
\bar\theta(x)\,=\,\left\{\bega{rl} \theta^-\quad\hbox{if}~~x<0,\\[1mm]
\theta^+\quad\hbox{if}~~x>0.\enda\right.
\eeq
A local solution to this generalized Riemann problem will be constructed
separately in  the four cases considered in Section~\ref{sec:3}.

\v
\textbf{Case 1}: $\theta^-=\theta^+=1$. 

In this case, we simply define  $u=u(t,x)$ to be the solution of the conservation law
(\ref{clawf}), with $\bar u$ as initial data.  
As for the Riemann problem, two subcases must be considered.

{\bf Case 1A:}   $u^-\leq  u^+$.
The solution of (\ref{clawf})  is thus monotone increasing,  and 
we can set 
$$\theta(t,x)=1\qquad \hbox{for all}~~ t,x.$$   
In this case, no interfaces are present. 

\medskip

{\bf Case 1B:} $u^->u^+$.  In this case the solution $u$ 
has a single downward jump, say located at $x=\gamma(t)$,  that satisfies the Lax condition \eqref{Lax}. 
Here the jump is an interface. We set 
$$\theta(t,x)=\begin{cases}1 &\hbox{if}~~ x\not = \gamma(t)\,,\\
0 &\hbox{if}~~  x=\gamma(t)\,.
\end{cases}$$
\v
\textbf{Case 2}: $\theta^-=1$, $\theta^+=0$.   
We begin by constructing two solutions:
\begi
\item The solution $u=u^\flat(t,x)$  of the Cauchy problem
\bel{cp21}
u_t + f(u)_x~=~0,\qquad u(0,x)\,=\,\left\{ \bega{rl}\bar u(x)\quad\hbox{if}~~x<0,\\[1mm]
+\infty\quad\hbox{if}~~x>0.\enda\right.\eeq
\item The solution $u=u^\sharp(t,x)$ of the Cauchy problem
\bel{cp22}
u_t + g(u)_x~=~0,\qquad 
\qquad 
u(0,x)\,=\,\left\{ \bega{cl}\bar u(x)\quad&\hbox{if}~~x>0,\\[1mm]
u^+ ~&\hbox{if}~~x<0.\enda\right.\eeq
\endi
Since $\bar u$ is increasing for $x<0$ and decreasing for $x>0$, the solution $u^\flat$ 
will contain only rarefaction fronts, together with a rarefaction fan centered at the origin, 
where $u$ ranges from $\bar u(0-)=u^-$ to $+\infty$. 
On the other hand, $u^\sharp$ will contain shocks and compression waves, see Fig~\ref{f:df18}.
%

The solution $u=u(t,x)$ of the original Cauchy problem will be obtained 
by gluing together these two solutions, namely
\bel{glu} u(t,x)~=~\left\{\bega{rl} u^\flat (t,x)\quad\hbox{for} ~~x<y(t),\\[1mm]
u^\sharp (t,x)\quad\hbox{for} ~~x>y(t),\enda\right.\eeq
for a suitable interface $y=y(t)$ which must be determined (see again Fig.~\ref{f:df18}).
Consider the function
\bel{Hdef}
H(t,x)~\doteq~
\frac{f\left( u^\flat (t, x-)\right) - g\left( u^\sharp(t, x+)\right) }{u^\flat (t, x-)- u^\sharp(t, x+)}\,.
\eeq
By the Rankine-Hugoniot equation (\ref{RH}), the interface must satisfy
\bel{odey}
\dot y(t)~=~H(t, y(t)),\qquad \qquad y(0)=0\,.\eeq
The heart of the matter is thus reduced to proving that 
the ODE with discontinuous coefficients \eqref{odey}
has a unique solution. 

As in Fig.~\ref{f:case2A} (left), let $(u^*, f(u^*))$ be the point where the straight line
through $(u^+, g(u^+))$ touches the graph of $f$ tangentially, with $u^*>u^+$.
Since the characteristics impinge on the shock at $y(t)$ from both sides, and $u^\flat(t,\cdot)$
is increasing while $u^\sharp(t,\cdot)$ is decreasing, the left and right states
\bel{ufs}
t\,\mapsto \, u^\flat\bigl(t, y(t)\bigr),\qquad\qquad t\,\mapsto \, u^\sharp\bigl(t, y(t)\bigr),\qquad\qquad t\in [0,t_0]
\eeq
are both non-increasing in time. Therefore they are BV functions, and so is the function 
$t\mapsto \dot y(t)=H\bigl(t, y(t)\bigr)$. In particular, the limit
$$\dot y(0+)\doteq\lim_{t\to 0+} H\bigl(t, y(t)\bigr)$$
is well defined.

\begin{figure}[htbp]
\centering{  \includegraphics[scale=0.35]{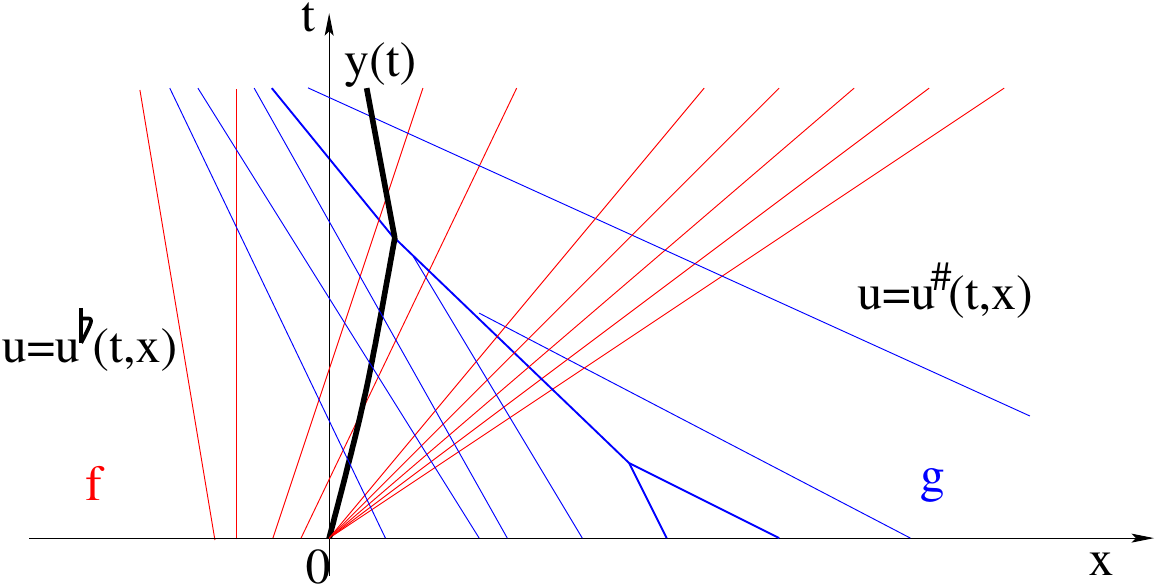}\qquad
  \includegraphics[scale=0.35]{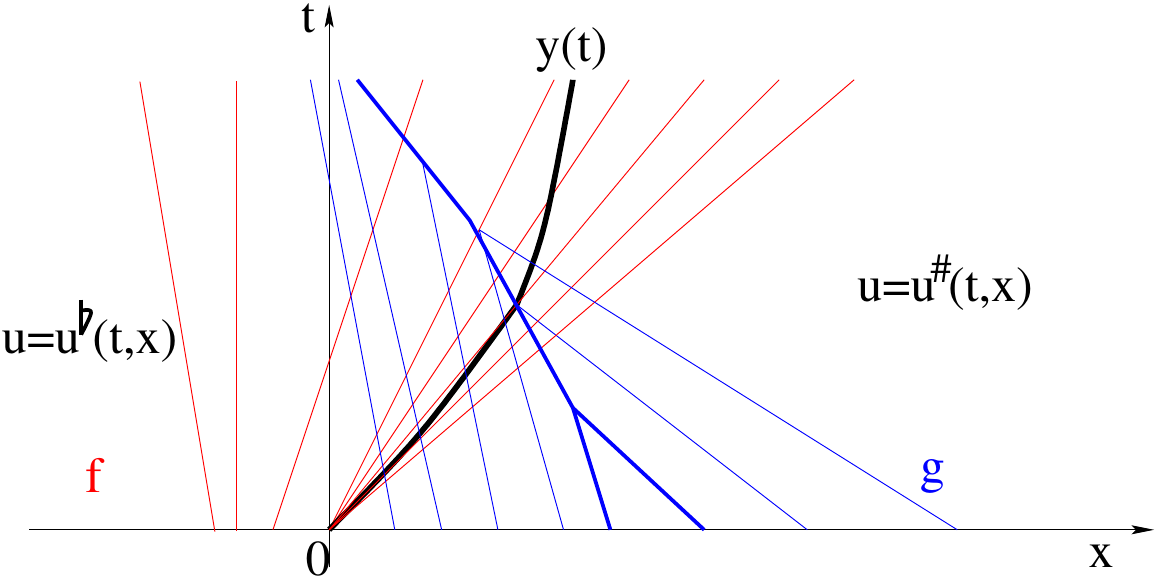}}
    \caption{\small   The solutions $u^\flat$ and $u^\sharp$ of (\ref{cp21}) and (\ref{cp22}), respectively,
    and the interface $y(\cdot)$.  Left: the case 2A where characteristics impinge on the interface $y(\cdot)$
    transversally from both sides.  Right: the case 2B, where the interface lies inside a centered rarefaction fan
    of $u^\flat$.  }
    \label{f:df18}
\end{figure}
%

As in Section~\ref{sec:3}, two sub-cases need to be considered.
\v
{\bf Case 2A}:  $u^->u^*$ (see Fig.~\ref{f:df18}, left).   
{\wen 
Note that, in this case, the solution of the Riemann problem 
(see  Section \ref{sec:3}, CASE 2A) is a 
shock. Characteristics impinge on the shock transversally both from the left and
from the right. 
By continuity, for the local solutions of \eqref{odey}, this strict transversality property continues to hold.}
We begin by proving    that
\bel{doty1}
\dot y(0+)~=~\lambda~\doteq~\frac{f(u^-)- g(u^+)}{ u^--u^+}~<~f'(u^-)\,.
\eeq
Indeed, for $u\geq u^-$ we have
$$\frac{f(u)-g(u^+)}{ u-u^+} ~<~f'(u).$$
On the region where $x/t\geq f'(u^-)$,  hence $u^\flat\geq u^-$, 
and for $g(u^\sharp)$ sufficiently close to $g(u^+)$, 
we have
$$\dot y(t)~=~H(t, y(t))~<~f'\bigl(u^\flat(t,x)\bigr)~=~\frac{x}{ t}\,.$$
This implies that the interface cannot enter the region where $x>t f'(u^-) $.
We can thus study the ODE (\ref{odey}) restricted to the domain where
\bel{ywe} y(t)~\leq ~t f'(u^-).\eeq
Observing that
$$\lim_{ (t,x)\to (0,0), ~  x\leq t f'(u^-)} H(t,x) ~=~\lambda,$$
our claim (\ref{doty1}) is proved. 
\v

Relying on (\ref{ywe}), for any $\ve_1>0$ we can find $\delta_1>0$ such that
\begin{align*}
\bigl|u^\flat (t,x)-u^-\bigr|<\ve_1 \quad &\hbox{for} ~~~x/t ~\leq f'(u^-),~~|t|+|x|<\delta_1,
\\[1mm]
\bigl|u^\sharp(t,x)- u^+\bigr|<\ve_1\quad & \hbox{for} ~~~|t|+|x|<\delta_1.
\end{align*}
Then, for any $\ve>0$ we can find $\delta>0$ such that the function in (\ref{Hdef}) satisfies
\bel{Fep}
\bigl|H(t,x)-\lambda\bigr|\,<\,\ve\quad\quad \hbox{whenever}\quad 
 {x/t}\leq f'(u^-),~~|t|+|x|<\delta.\eeq
We now consider the cone 
\bel{coga}\Gamma~=~\bigl\{(t,x)\,;~~t\geq 0,~~|x-\lambda t|\leq \ve t\bigr\}.\eeq
By (\ref{Fep}), all solutions to the Cauchy problem~\eqref{odey}
remain inside $\Gamma$, for $t\in [0, t_0]$ small.


We claim that, in the cone $\Gamma$, 
{\wen thanks to the strict transversality property,}
the function $H$ has bounded directional variation, hence the uniqueness
theorem in \cite{B88} can be applied.
Indeed, let $t\mapsto x(t)$ be any Lipschitz function such that
$$\bigl|H(t,x)-\lambda\bigr|~\leq~2\ve\qquad\qquad\forall t\in [0,t_0],$$
with $t_0>0$ sufficiently small.
Since in (\ref{Hdef}) the denominator is uniformly positive on $\Gamma$,
it suffices to check that the total variation of the functions
\bel{ufs2}
t\,\mapsto \, u^\flat\bigl(t, x(t)\bigr),\qquad\qquad t\,\mapsto \, u^\sharp\bigl(t, x(t)\bigr),\qquad\qquad t\in [0,t_0],
\eeq
is uniformly bounded.
This is clear, because on $\Gamma$ the characteristic speeds satisfy
$$
f'\bigl( u^\flat(t, x(t))\bigr)~>~\lambda+2\ve,\qquad\qquad g'\bigl( u^\sharp(t, x(t))\bigr)~<~\lambda-2\ve.
$$
Hence the two functions  in (\ref{ufs2}) are both non-increasing. Their total variation is bounded in terms
of the total variation of $\bar u$ in a neighborhood of the origin.  
An application of the result in \cite{B88} now yields the local existence of a unique solution $t\mapsto y(t)$
to the 
Cauchy problem (\ref{odey}).

\v

{\bf Case 2B}:  $u^-\leq u^*$   (see Fig.~\ref{f:df18}, right).
{\wen 
We recall that the solution of the Riemann problem 
(see Section~\ref{sec:3}, CASE 2B) consists of a centered rarefaction attached to the left of a shock.
The speed of the shock coincides with the speed of characteristics on the left of the shock. 
In this case, the strict transversality property no longer holds
and the result in \cite{B88} cannot be directly applied. 
Uniqueness of the solution to the ODE   (\ref{odey}) will be proved by means of the auxiliary Lemma~\ref{l:41}, stated below.}

We first observe that, by (\ref{cp22}), 
$$\lim_{(t,x)\to (0,0)}~u^\sharp(t,x)~=~u^+.$$
With reference to Fig.~\ref{f:case2B}, left, we have the implications
\bel{we1}\left\{\bega{l}\ds u>u^*\quad\implies\quad f'(u)~>~\frac{f(u)-g(u^+)}{ u-u^+}\,,
\\[2mm] \ds u<u^*\quad\implies\quad f'(u)~<~\frac{f(u)-g(u^+)}{ u-u^+}\,.\enda\right.\eeq
Let now $\ve_2>0$ be given and be sufficiently small. Consider the wedge
\bel{wedge}
{\cal W}_{\ve_2}~\doteq~\bigl\{ (t,x)\,;\quad t\ge 0,\quad (\lambda-\ve_2)t \leq {x}\leq 
(\lambda+\ve_2)t\bigr\}.
\eeq
Therefore, 
for $t>0$ and all points $(t,x)$ sufficiently close to the origin,  we have 
\bel{we2}\left\{\bega{l}\ds  f' (u^\flat) \geq \lambda+\ve_2 \quad\implies\quad \frac{x}{ t}
\,=\,f'(u^\flat)\,>\,\frac{f(u^\flat)-g(u^\sharp)}{ u^\flat-u^\sharp}\,,\\[2mm]
\ds  f' (u^\flat) \leq \lambda-\ve_2  \quad\implies\quad \frac{x}{ t} \,=\,f'(u^\flat )\,<\,
\frac{f(u^\flat )-g(u^\sharp)}{ u^\flat-u^\sharp}\,.\enda\right.
\eeq
This shows that, for $t\in [0, t_\ve]$ small, the trajectory of (\ref{odey}) remains inside the wedge ${\cal W}_{\ve_2}$.
\v
 Working within the class of Lipschitz functions 
whose graph is contained in ${\cal W}_{\ve_2}$, we now consider the Picard operator
\bel{Pic}
(\P y)(t)=\int_0^t 
\phi\left(u^\flat\bigl(s,y(s)), u^\sharp(s,y(s)+)\right)\, ds ,
\qquad\mbox{where}\quad
\phi(u,v) \,\dot=\, \frac{f(u)-g(v)}{u-v}. 
\eeq
Since in the wedge ${\cal W}_{\ve_2}$ we have $u^* > u^+ > u^\sharp$, 
the denominator of 
the integrand in the above Picard operator remains uniformly positive.
Therefore  the integrand function $\phi=\phi(u,v)$ depends Lipschitz continuously on its arguments. 

In order to achieve a contractivity property,  consider the family of Lipschitz functions 
\bel{FLIP}\F~\doteq~\left\{ y\in W^{1,\infty}\bigl([0,t_0]\bigr);~\bigl|\dot y(t)-f'(u^*)\bigr|\leq \ve_2~~\hbox{for a.e.}~
t,\quad
\lim_{t\to 0+} \frac{y(t)}{ t}=f'(u^*)\right\},\eeq
with distance
$$d(y,z)~\doteq~\sup_{0<t\leq t_0} \left| \frac{y(t)-z(t)}{ t}\right|.$$
Let $y,z$ be two such Lipschitz functions with  $d(y,z) =\delta$,
so that (see Fig.~\ref{f:df30})
\bel{yz} 
\bigl|y(t)-z(t)\bigr|~\leq~\delta t\qquad\qquad\forall t\in [0,t_0].
\eeq

To estimate the distance $\Big| (\P y)(t)- (\P z)(t)\Big|$, for clarity we estimate separately 
the terms arising from the change in $u^\flat$ and the terms arising from the change in $u^\sharp$. 
We  write 
\bel{P11} 
 (\P y)(t)- (\P z)(t)~=~ \int_0^tA(s)\, ds+\int_0^t B(s)\, ds\,,
\eeq
where
\begin{align}
A(s) &~ \dot= ~\phi\left(u^\flat\bigl(s,y(s)), u^\sharp(s,y(s)+)\right) - \phi\left(u^\flat\bigl(s,z(s)), u^\sharp(s,y(s)+)\right) ,
\label{eq:A}\\
B(s)  & ~\dot=~ \phi\left(u^\flat\bigl(s,z(s)), u^\sharp(s,y(s)+)\right) - \phi\left(u^\flat\bigl(s,z(s)), u^\sharp(s,z(s)+)\right).
\label{eq:B}
\end{align}
{\wen
 Notice that the  term $A$ in \eqref{eq:A} 
accounts for the difference in speed $\dot y - \dot z$, due to
different values of $u^\flat$ on the left  of the interface.
}
Applying the mean value theorem we obtain
\[
A(s)\, =\, \phi_u \left(\Hat u^\flat, u^\sharp(s,y(s)+)\right) \cdot 
\left( u^\flat\bigl(s,y(s)) - u^\flat\bigl(s,z(s))\right) ,
\]
where $\Hat u^\flat$ is some intermediate value between $u^\flat\bigl(s,y(s))$ and $u^\flat\bigl(s,z(s))$.
Note that
\begin{align}
\label{phiu0}
\phi_u(u^*,u^+)\,&=\, \frac{1}{u^*-u^+} \left(f'(u^*) - \frac{f(u^*)-g(u^+)}{u^*-u^+}\right)\,=\, 0.
\end{align}
Performing a linearization of the term $\phi_u$ at the point  $(u^*,u^+)$,  and using \eqref{phiu0}, 
we obtain
\begin{align}
\bigl|A(s)\bigr|
 &\le
\O(1) \cdot
 \Big(\bigl|u^*- u^\flat(s,y(s))\bigr| +\bigl|u^*- u^\flat(s,z(s))\bigr|+ \bigl| u^+ - u^\sharp(s,y(s))\bigr| \Big) \cdot   |y(s)-z(s)| 
\nonumber\\
  &\leq  \O(1) \cdot s\, d(y,z),
\label{Pyz} 
\end{align}
The terms $ |u^*- u^\flat|$, $ \bigl| u^+ - u^\sharp|$ can be rendered arbitrarily small by choosing
$\ve_2>0$ small enough in the definition of the wedge ${\cal W}_{\ve_2}$ at (\ref{wedge}).
This yields, for some constant $C$,
\bel{Abo}\int_0^t \bigl|A(s)\bigr|\, ds~\leq~C\cdot \frac{t^2}{ 2}\, d(y,z)~\leq~\frac{t}{ 4} \,d(y,z)\eeq
for all $t\in [0,\tau]$, provided $\tau>0$ was chosen small enough.

%
%
{\wen 
It remains to estimate the integral of $B$  in \eqref{P11}.
Notice that the  term $B$ in \eqref{eq:B} 
accounts for the difference in speed $\dot y - \dot z$, due to
different values of $u^\sharp$ on the right  of the interface.
Toward this goal, we shall use }

\begin{lemma}\label{l:41} Given a constant $a>0$, let $v:[0,\tau]\times \R\mapsto\R$ be a function with the properties
\begi
\item[(i)] For every time $t$, the function $x\mapsto v(t,x)$ is decreasing.
\item[(ii)] For every Lipschitz curve $t\mapsto \gamma(t)$ 
with derivative $\dot\gamma(t)\geq a-\ve$, the composite function
$t\mapsto v\bigl(t, \gamma(t)\bigr)$ is decreasing.
\endi
Consider any two Lipschitz functions $y,z:[0,\tau]\mapsto \R$ such that  
$$y(0)=z(0),
\qquad\qquad \dot y(t)\geq a,\quad \dot z(t)\geq a\quad\hbox{ for a.e.}~t\in [0,\tau],$$
and define
$$v_{max} \,\doteq\,\sup_{t,x} ~v(t,x),\qquad v_{min}\,\doteq \,\inf_{t,x} ~v(t,x),\qquad\qquad  \delta(t)\,\doteq\,\sup_{0<s<t} 
\bigl|y(s)-z(s)\bigr|.$$
Then, for every $t\in [0,\tau]$ one has
\bel{iclo}
\int_0^t \bigl| v(s, y(s))- v(s, z(s))\bigr|\, ds~\leq~2\delta(t)\cdot \frac{v_{max}-v_{min}}{\ve}\,.\eeq
\end{lemma}

\begin{figure}[htbp]
\begin{center}
\resizebox{.6\textwidth}{!}{
\begin{picture}(0,0)%
\includegraphics{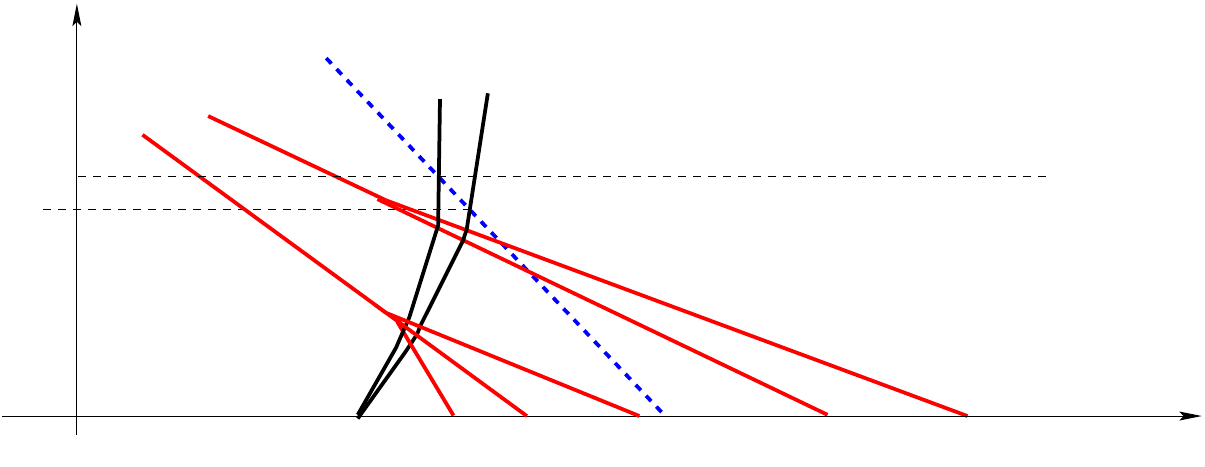}%
\end{picture}%
\setlength{\unitlength}{3947sp}%
\begin{picture}(9627,3717)(1186,434)
\put(10426,539){\makebox(0,0)[lb]{\smash{\fontsize{20.74}{26.4}\usefont{T1}{ptm}{m}{n}{\color[rgb]{0,0,0}$x$}%
}}}
\put(1351,2714){\makebox(0,0)[lb]{\smash{\fontsize{20.74}{26.4}\usefont{T1}{ptm}{m}{n}{\color[rgb]{0,0,0}$t$}%
}}}
\put(1201,2339){\makebox(0,0)[lb]{\smash{\fontsize{20.74}{26.4}\usefont{T1}{ptm}{m}{n}{\color[rgb]{0,0,0}$t'$}%
}}}
\put(4351,3464){\makebox(0,0)[lb]{\smash{\fontsize{20.74}{26.4}\usefont{T1}{ptm}{m}{n}{\color[rgb]{0,0,0}$y(t)$}%
}}}
\put(5026,3464){\makebox(0,0)[lb]{\smash{\fontsize{20.74}{26.4}\usefont{T1}{ptm}{m}{n}{\color[rgb]{0,0,0}$z(t)$}%
}}}
\put(3226,3839){\makebox(0,0)[lb]{\smash{\fontsize{20.74}{26.4}\usefont{T1}{ptm}{m}{n}{\color[rgb]{0,0,1}$\gamma(t)$}%
}}}
\end{picture}%
}
\caption{\small Proving Lemma~\ref{l:41}.
The function $v=v(t,x) $ is decreasing on every horizontal line. It is also decreasing (as a function of time)
along the line $\gamma$.  Hence $v(t, y(t)) \leq v(t', z(t'))$ for some time $t'$ satisfying (\ref{zt'}).
}
\label{f:df30}
\end{center}
\end{figure}

\begin{figure}[htbp]
\begin{center}
\resizebox{.7\textwidth}{!}{
\begin{picture}(0,0)%
\includegraphics{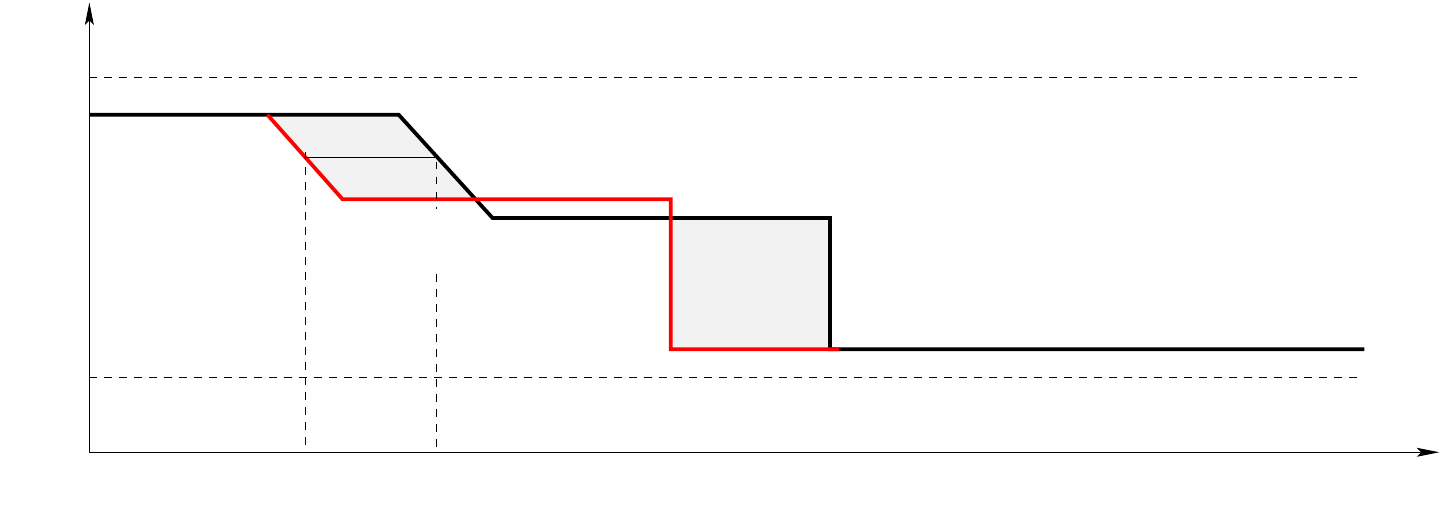}%
\end{picture}%
\setlength{\unitlength}{3947sp}%
\begin{picture}(11527,4027)(-714,424)
\put(1676,2454){\makebox(0,0)[lb]{\smash{\fontsize{20.74}{26.4}\usefont{T1}{ptm}{m}{n}{\color[rgb]{1,0,0}$v(t,z(t))$}%
}}}
\put(2886,3169){\makebox(0,0)[lb]{\smash{\fontsize{20.74}{26.4}\usefont{T1}{ptm}{m}{n}{\color[rgb]{0,0,0}$v(t,y(t))$}%
}}}
\put(1591,529){\makebox(0,0)[lb]{\smash{\fontsize{20.74}{26.4}\usefont{T1}{ptm}{m}{n}{\color[rgb]{0,0,0}$t'$}%
}}}
\put(-684,3769){\makebox(0,0)[lb]{\smash{\fontsize{20.74}{26.4}\usefont{T1}{ptm}{m}{n}{\color[rgb]{0,0,0}$v_{\max}$}%
}}}
\put(-699,1354){\makebox(0,0)[lb]{\smash{\fontsize{20.74}{26.4}\usefont{T1}{ptm}{m}{n}{\color[rgb]{0,0,0}$v_{\min}$}%
}}}
\put(2756,524){\makebox(0,0)[lb]{\smash{\fontsize{20.74}{26.4}\usefont{T1}{ptm}{m}{n}{\color[rgb]{0,0,0}$t$}%
}}}
\put(10116,574){\makebox(0,0)[lb]{\smash{\fontsize{20.74}{26.4}\usefont{T1}{ptm}{m}{n}{\color[rgb]{0,0,0}$t$}%
}}}
\end{picture}%
}
\caption{\small  The estimate (\ref{ppart}).    For every time $t$ such that $v(t, y(t)) > v(t, z(t))$, 
one can find an earlier time $t'$ such that (\ref{tt'})-(\ref{zt'}) hold.
}
\label{f:df45}
\end{center}
\end{figure}

{\bf Proof of Lemma~\ref{l:41}.}  By the assumption (ii), the two maps $t\mapsto v(t, y(t))$ and $t\mapsto v(t,z(t))$ are decreasing.
For any time $t\in [0,\tau]$, two cases can arise:

Case 1: $y(t)\geq z(t)$.  By (i) this yields $v(t, y(t))\leq v(t, z(t))$.

Case 2: $y(t)< z(t)$. In this case, consider the line $\gamma(s)= y(t) + (a-\ve) (s-t)$. Let $t'$ be the time 
where $ z(t') ~=~\gamma(t')$, see Figure~\ref{f:df30}. 
Observing that
$$\dot\gamma(s)~=~a-\ve~\leq~\dot z(s)-\ve,$$
we have
\bel{tt'} 0~<~t-t'~\leq~\frac{z(t)- y(t)}{\ve}~\leq~\frac{\delta(t)}{ \ve}.\eeq
Moreover, since $s\mapsto v\bigl(s, \gamma(s)\bigr)$ is decreasing, one has
\bel{zt'} v(t, y(t))~\leq~ v(t', z(t')).\eeq

Denoting by $[\alpha]_+=\max\{\alpha,0\}$ the positive part of a number $\alpha$, the previous inequalities yield
{\wen (see the shaded area in Figure~\ref{f:df45})}
\begin{eqnarray}
&& \hspace{-2cm} 
\int_0^t \Big[v\bigl(s, y(s)\bigr)- v\bigl(s, z(s)\bigr)\Big]_+ds
\nonumber 
\\
&=&
{\wen \int_0^t\meas\Big\{ \omega\in\R;~v(s,y(s))>\omega >v(s, z(s))\Big\}\, ds}
\nonumber
\\
&=& \int_{v_{min}}^{v_{max}} \meas\Big\{ s\in [0,t]\,;~v(s,y(s))>\omega >v(s, z(s))\Big\}\, d\omega
\nonumber
\\
&\leq& 
(v_{max}-v_{min}) \, \frac{\delta(t)}{\ve}\,.
\label{ppart}
\end{eqnarray}
%
Notice that the last inequality follows from (\ref{tt'}).

Switching the roles of $y$ and $z$, we obtain a similar inequality for the positive part of $v(s,z(s))-v(s,y(s))$.
Adding up these estimates and using \eqref{zt'}, we get \eqref{iclo}, 
proving the lemma. \endproof

%

\v
Resuming the proof of Theorem~\ref{t:41}, to estimate the term $B$ in (\ref{eq:B}) we consider the function
\bel{phide}
 (s,w)\mapsto \Phi(s, w)~\doteq~\phi\bigl(u^\flat\bigl(t,z(s)), w\bigr) ~= ~
\frac{f\bigl(u^\flat(t,z(s))\bigr) - g(w)}{ u^\flat\bigl(t,z(s)\bigr) - w }\,, \quad 0<s<t\,.
\eeq
Notice that $\Phi$ is Lipschitz continuous w.r.t.~$w$, 
say with constant $L$. Moreover, the function $u^\sharp(t,x)$
defined at (\ref{cp22}) satisfies all the assumptions imposed on $v$ in Lemma~\ref{l:41}, for some  $\ve>0$ providing 
a lower bound on the difference between the characteristic speeds.

Using Lemma~\ref{l:41} we can thus estimate the second integral in (\ref{P11}) as
\begin{align}
\ds\int_0^t \bigl| B(s)\bigr|\, ds \;
&=\ds~\int_0^t \Big|\Phi\bigl(s,u^\sharp(s,y(s))\bigr)- \Phi\bigl(s,u^\sharp(s,z(s))\bigr)\Big|\, ds
\nonumber\\
&\leq~\ds L\cdot \int_0^t \Big| u^\sharp(s, y(s)) - u^\sharp(s, z(s))\Big|\, ds
\nonumber\\
&\leq \ds~\frac{L}{\ve} \cdot \bigl(u^\sharp_{max}-u^\sharp_{min}\bigr)\cdot \sup_{0<s<t} \bigl|y(t)-z(t)\bigr|
\nonumber\\
&\leq~\ds \frac{L}{\ve} \cdot \bigl(u^\sharp_{max}-u^\sharp_{min}\bigr)\cdot  t \, d(y,z).
\label{Bbound}
\end{align}
We can now choose $\tilde x>0$ small enough so that the total variation of $\bar u = u^\sharp(0,\cdot)$ 
on $\,]0, \tilde x]$ is as small as we like.
In particular, working  on a suitably small time interval $[0,\tau]$, we can assume
$$u^\sharp_{max}-u^\sharp_{min}~\leq~\frac{\ve}{ 4L}\,.$$
For every $t\in [0,\tau]$, this yields
\bel{BBB}
\int_0^t \bigl| B(s)\bigr|\, ds~\leq~\frac{1}{ 4} \, t\, d(y,z).\eeq
Combining (\ref{Abo}) with (\ref{BBB}) we obtain
$$\left|\int_0^t A(s) \,ds\right|+\left|\int_0^t B(s) \,ds\right|~\leq~\frac{1}{ 2}  \,t\, d(y,z).$$
This shows that the Picard operator is
a strict contraction, hence it admits a unique fixed point.
This completes the analysis for \textbf{Case 2B}.

\v
The remaining cases can be handled by entirely similar techniques.

\textbf{Case 3}: $\theta^-=0$, $\theta^+=1$.   
We then consider two solutions:
\begi
\item The solution $u=u^\flat(t,x)$  of the Cauchy problem
\bel{cp3}
u_t + g(u)_x~=~0,\qquad u(0,x)\,=\,\left\{ \bega{rl}\bar u(x)\quad\hbox{if}~~x<0,\\[1mm]
u^-\quad\hbox{if}~~x>0.\enda\right.\eeq
\item The solution $u=u^\sharp(t,x)$ of the Cauchy problem
\bel{cp4}
u_t + f(u)_x~=~0,\qquad u(0,x)\,=\,\left\{ \bega{rl}\bar u(x)\quad\hbox{if}~~x>0,\\[1mm]
-\infty \quad\hbox{if}~~x<0.\enda\right.\eeq
\endi
Since $\bar u$ is decreasing for $x<0$ and increasing for $x>0$, the solution
$u^\sharp$ will contain only rarefaction fronts, while $u^\flat$ will contain
shocks and compression waves.   

The solution $u=u(t,x)$ of the original Cauchy problem is obtained 
by gluing together these two solutions, as in (\ref{glu}),
for a suitable interface $y=y(t)$ which must be determined. By the Rankine-Hugoniot
equation (\ref{RH}), this interface must satisfy the ODE~\eqref{odey} 
where now the function $H$ in~\eqref{Hdef} is replaced with
\bel{Hdef2}
H(t, y(t))~\doteq~\frac{g\bigl( u^\flat (t, y(t)-)\bigr) - f\bigl( u^\sharp(t, y(t)+)\bigr) }{
u^\flat (t, y(t)-)- u^\sharp(t, y(t)+)}\,.\eeq

The proof of existence and uniqueness of the solution to the ODE \eqref{odey}, \eqref{Hdef2} 
is achieved by the same arguments as in 
\textbf{Case 2}. We thus omit the details.

\v
\textbf{Case 4}: $\theta^-=\theta^+=0$.   Two sub-cases must be considered. 

{\bf Case 4A}: $\theta^-=\theta^+=0$  and  $u^-\geq u^+$.
In this case, a solution to the Riemann 
problem is obtained by letting  $u=u(t,x)$ be the solution to the conservation law
(\ref{clawg}) 
with $\bar u$ as initial data, and taking $\theta(t,x)=0$ for all $t,x$.
\v
{\bf Case 4B}: $\theta^-=\theta^+=0$  and $u^-< u^+$.
{\wen Motivated by the solution of the Riemann problem in Section \ref{sec:3}, 
in  this case  we have 2 interfaces generated by this type of discontinuity at $t=0$.
Note that this  could only occur at initial time $t=0$.   
For the interface on the right, it can be treated in the same way as for \textbf{Case 2B},
while the interface on the left is similar to that of \textbf{Case 3B}.}


We  consider two solutions:   
\begi
\item The solution $u=u^\flat(t,x)$  of the Cauchy problem
\bel{cp5}
u_t + g(u)_x~=~0,\qquad u(0,x)\,=\,\left\{ \bega{rl}\bar u(x)\quad\hbox{if}~~x<0,\\[1mm]
u^-\quad\hbox{if}~~x>0.\enda\right.\eeq
\item The solution $u=u^\sharp(t,x)$ of the Cauchy problem
\bel{cp6}
u_t + g(u)_x~=~0,\qquad u(0,x)\,=\,\left\{ \bega{rl}\bar u(x)\quad\hbox{if}~~x>0,\\[1mm]
u^+ \quad\hbox{if}~~x<0.\enda\right.\eeq
\endi
Notice that both of these solutions are monotone decreasing, and can contain only
shocks or compression waves.
In addition, we consider the solution to (\ref{clawf}) containing a single, infinitely large 
centered rarefaction. This can be implicitly defined by setting
\bel{Tudef} 
u^\natural (t,x)\,=\,w\qquad \hbox{if}\qquad  
x/t = f'(w).\eeq
As in CASE 4B of the Riemann problem discussed in Section~\ref{sec:3}, we construct a solution to the Cauchy problem by setting
\bel{C4} 
u(t,x)~=~\left\{\bega{rl} u^\flat(t,x)\quad &\hbox{if}\quad x<y(t),\\[1mm]
u^\natural (t,x)\quad &\hbox{if}\quad y(t)<x<z(t),\\[1mm]
u^\sharp(t,x)\quad &\hbox{if}\quad z(t)<x\,.\enda\right.\eeq
The two interfaces $y(\cdot)$ and $z(\cdot)$ are 
uniquely determined by solving the ODEs
\begin{align*}
&\dot y(t)~=~
\frac{g\bigl( u^\flat(t, y(t)-)\bigr) - f\bigl(
u^\natural(t,y(t)+)\bigr)}{u^\flat(t, y(t)-) -
u^\natural (t,y(t)+)}\,,\qquad\qquad y(0)=0,\\
&\dot z(t)~=~
\frac{f\bigl( 
u^\natural(t, z(t)-)\bigr) - g\bigl(u^\sharp(t,z(t)+)\bigr)}{
u^\natural(t, z(t)-) - u^\sharp (t,z(t)+)}\,,\qquad\qquad z(0)=0.
 \end{align*}
The existence and uniqueness of the solutions $y$ and $z$ is proved 
by the same arguments as  {\wen in \textbf{Case 3B}  and \textbf{Case 2B}, respectively. }

\subsection{Global solutions} \label{S4-2}
By the previous steps, a unique solution $u=u(t,x)$ can  
be constructed on some initial interval  $[0, t_0]$.  
{\wen These include all cases discussed in previous sub-section, 
including Case 1A and Case 4B, which 
could only occur at initial time $t=0$.  
Note that at  the initial time $t=0+$ both the number of interfaces and the  total variation of the solution 
can increase compared to the initial data, as in Case 4B, but they remain finite. }

We now claim that this solution can be uniquely extended up to the first time $\tau_1$ where two of the interfaces 
$y_i(\cdot)$ meet each other.

Indeed, on a neighborhood of  a point where $\theta(t_0,\cdot)$ is constant, 
the solution can be prolonged in time simply by constructing an entropy-admissible solution to the
conservation law (\ref{clawf}) if $\theta=1$ or the conservation law (\ref{clawg}) if $\theta=0$,
{\wen as in \textbf{Case 1}}. 

It remains to show that the solution can be extended also in a neighborhood of 
a point $y_i(t)$ where $\theta$ has a jump.  
To fix ideas, assume $\theta(t_0,x)=1$ for $x<y_i(t_0)$ and $\theta(t_0,x)=0$ for $x>y_i(t_0)$.   
In this case, a unique solution is constructed as in {\bf Case 2} of the previous analysis.
On the other hand, if $\theta(t_0,x)=0$ for $x<y_i(t_0)$
and $\theta(t_0,x)=1$ for $x>y_i(t_0)$, the construction performed in  \textbf{Case 3} 
applies.  

In both cases, the solution is prolonged in time, keeping constant the number of interfaces $y_i(\cdot)$.
We observe that, by the Lax admissibility condition (\ref{Lax}), 
characteristics impinge on the curve $x=y_i(t)$ from both sides. As a consequence,
the total variation of the solution $u(t,\cdot)$ does not increase in time.
{\wen
By the same argument, the $\L^\infty$ norm of the
solution does not increase.
}

{\wen
This construction can be continued up to
the first time $\tau_1$ where two or more interfaces join together.   In this case, 
the solution is restarted.   
 As noted before, 
\textbf{Case 4B} will not occur at such interactions.

Therefore, at every restarting point,  the  number of interfaces decreases at least by one.
Hence the total number of these restarting times is finite.  In finitely many steps,
a unique global in time  solution is obtained, for all $t\ge 0$.  
Since the number of restarting times is finite, the solution is uniformly bounded.

To complete the proof of Theorem~\ref{t:41}, it remains to observe that our piecewise monotone solution
is continuous as a map from $[0,T]$ into $\L^1_{loc}$.   Indeed, restricted to a region
 $$\bigl\{ (t,x)\,;~~y_{k-1}(t)< x < y_k(t)\bigr\}$$
 between two interfaces, our solution coincides with an entropy solution of one of the conservation laws 
(\ref{fgclaw})
Since there are finitely many of these interfaces, and their speeds $\dot y_k(t)$ are uniformly bounded,
the conclusion is clear.
}
\endproof

\section{Concluding remarks}
\label{sec:6}
\setcounter{equation}{0}
In this paper we considered the unstable case where $f(u)>g(u)$ for all $u\in\R$, 
assuming that both $f$ and $g$ are strictly convex.  Entirely similar results
can be proved under the assumption that both $f$ and $g$ are strictly concave.
This is a situation usually encountered in connection with traffic flow.

{\wen
The stable case, where $f(u)<g(u)$ for all $u\in\R$, has been studied in the companion paper \cite{ABS}.  
At this stage, a natural open question is what happens if the graphs of $f$ and $g$ intersect at one point.
Indeed, in connection with  models of traffic flow with hysteresis, observed traffic data suggests  that
for low density traffic one has $f<g$, while for high density traffic one has $f>g$; see~\cite{TM_data, CF19}. 
As shown by the examples in Section~\ref{sec:2}, this instability allows the emergence of new spikes
in the traffic density, 
as a response to arbitrarily small perturbations.     
Such spikes will then persist in time, leading to the formation of stop-and-go waves.  
A stochastic model, where the location of new spikes follows a Poisson 
distribution in space and time, is currently being investigated in \cite{BSu}.
}

\v
{\small
{\bf Acknowledgments.} The research of A.\,Bressan was partially supported by NSF with
grant  DMS-2306926, ``Regularity and approximation of solutions to conservation laws".
The research by D.\,Amadori was partially supported by the Ministry of University and Research (MUR), Italy 
under the grant PRIN 2020 - Project N. 20204NT8W4, ``Nonlinear evolution PDEs, fluid dynamics and transport equations: 
theoretical foundations and applications" and by the INdAM-GNAMPA Project 2023, CUP E53C22001930001, 
``Equazioni iperboliche e applicazioni". 
D.\,Amadori acknowledges the kind hospitality of Penn State University, where this research started.
{\wen The authors also wish to thank the anonymous referees, for carefully reading the manuscript
and for their several useful suggestions.}
}


\end{document}